\documentclass[a4paper,11pt,reqno,noindent]{amsart}
\usepackage[centertags]{amsmath}
\usepackage{amsfonts,amssymb,amsthm,dsfont,cases,amscd,esint,enumerate, stmaryrd}
\usepackage[T1]{fontenc}
\usepackage[english]{babel}
\usepackage[applemac]{inputenc}
\usepackage{newlfont}
\usepackage{color}
\usepackage[body={17cm,21.5cm}, top = 1in, bottom = 1.5in, centering]{geometry} 
\usepackage{fancyhdr}
\usepackage{enumerate}

\numberwithin{equation}{section}

\usepackage[colorlinks,citecolor=black,urlcolor=black]{hyperref}

\newtheorem{theorem}{Theorem}[section]
\newtheorem{proposition}[theorem]{Proposition}

\newtheorem{definition}[theorem]{Definition}
\newtheorem{lemma}[theorem]{Lemma}
\newtheorem{remark}[theorem]{Remark}
\newtheorem{example}[theorem]{Example}
\newcommand{\R}{{\mathbb{R}}}
\newcommand{\T}{{\mathbb{T}}}
\newcommand{\D}{{\mathbb{D}}}
\newcommand{\Dd}{{\mathbb{D}}}

\newcommand{\Pm}{{\mathbb{P}}}
\newcommand{\E}{{\mathbb{E}}}

\newcommand{\e}{{\varepsilon}}
\newcommand{\Id}{\operatorname{Id}}
\newcommand{\loc}{{\operatorname{loc}}}

\newcommand{\Sym}{{\operatorname{sym}}}
\newcommand{\Div}{{\operatorname{div}}}
\newcommand{\Pc}{\mathcal P}

\newcommand{\Hc}{\mathcal H}
\newcommand{\Mcal}{\mathcal M}
\newcommand{\eps}{\varepsilon}

\title[A duality method for mean-field limits]{A duality method for mean-field limits\\with singular interactions}

\author[D.~Bresch]{D. Bresch}
\address{\sc D. Bresch. Universit\'e Savoie Mont Blanc, UMR5127 CNRS, Laboratoire de Math\'ematiques, 73376 Le Bourget-du-Lac, France}
\email{didier.bresch@univ-smb.fr}

\author[M.~Duerinckx]{M. Duerinckx}
\address{\sc M. Duerinckx. Universit\'e Libre de Bruxelles, D\'epartement de Ma\-th\'e\-matiques, 1050 Brussels, Belgium}
\email{mitia.duerinckx@ulb.be}

\author[P.-E.~Jabin]{P.-E. Jabin}
\address{\sc P.-E. Jabin.  Department of Mathematics and Huck Institutes, Pennsylvania State University, State College, PA 16801, USA}
\email{pejabin@psu.edu}

\begin{document}

\begin{abstract}
In this paper, we introduce a new approach to derive mean-field limits for first- and second-order particle systems with singular interactions. It is based on a duality argument combined with the analysis of the hierarchy of dual correlations, and allows to treat for the first time arbitrary square-integrable interaction forces, including at vanishing temperature.
Notably, the method requires neither a sign condition nor a potential structure for the interaction. It also applies to alignment models, extending the range of singular interactions for which mean-field limits were previously available. A refined analysis of the dual hierarchy, exploiting kinetic regularization effects, is pursued in a companion paper and goes beyond the square-integrable threshold, leading in particular to the derivation of the Vlasov-Poisson equation for 2D Coulomb particles.
\end{abstract}

\maketitle
\setcounter{tocdepth}{1}
\tableofcontents
\allowdisplaybreaks

\section{Introduction}
Consider the classical Newton dynamics for $N$ indistinguishable point particles with pairwise interactions.
Letting $d\ge1$ be the space dimension, we denote by $X_{i,N}\in \Omega$ and $V_{i,N}\in\R^d$ the positions and velocities of the particles, labeled by $1\le i\le N$, where the space domain $\Omega\subset\R^d$ is either the whole space $\R^d$ or the periodic torus $\T^d$. The particle dynamics is governed by the following ODEs,
\[\left\{
\begin{array}{l}
\displaystyle\frac{d}{dt}X_{i,N}=V_{i,N},\\[1mm]
\displaystyle
\frac{d}{dt}V_{i,N}=\frac1{N-1}\sum_{j:j\ne i}^NK(X_{i,N}-X_{j,N}),\qquad 1\le i\le N,
\end{array}\right.\]
where $K:\Omega\to\R^d$ is an interaction force kernel satisfying $K\in L^1_\loc(\Omega;\R^d)$
and where the mean-field scaling is considered.\footnote{We use the prefactor $\frac1{N-1}$ instead of the more usual $\frac1N$ for convenience; this has no effect on the mean-field limit.}

As our results apply identically in the presence of velocity diffusion, we consider more generally the following system of SDEs,
\begin{equation}\label{eq:SDE}
\left\{\begin{array}{l}
\displaystyle
dX_{i,N}=V_{i,N}dt,\\[1mm]
\displaystyle
dV_{i,N}=\frac1{N-1}\sum_{j:j\ne i}^NK(X_{i,N}-X_{j,N})\,dt+\sqrt{2\alpha}\,dB_{i,N},\qquad 1\le i\le N,
\end{array}\right.
\end{equation}
where $\{B_{i,N}\}_{1\le i\le N}$ are $N$ independent Brownian motions and where the temperature $0\le\alpha<\infty$ is a fixed parameter. The deterministic dynamics corresponds to $\alpha=0$.

Switching to a statistical perspective, we consider a density~$F_N$ on the $N$-particle phase space $\D^N:=(\Omega\times\R^d)^N$, and the system of SDEs~\eqref{eq:SDE} then formally leads to the following Liouville equation,
\begin{equation}\label{eq:Liouville}
\partial_tF_N+\sum_{i=1}^N\Big(v_i\cdot\nabla_{x_i}F_N
+\frac1{N-1}\sum_{j: j\ne i}^NK(x_i-x_j)\cdot\nabla_{v_i}F_N\Big)
\,=\,\alpha\sum_{i=1}^N\Delta_{v_i}F_N.
\end{equation}
The exchangeability of the particles amounts to assuming that $F_N$ is symmetric with respect to the variables
\[z_i=(x_i,v_i) \in\D:=\Omega\times\R^d, \qquad 1\le i \le N.\]
We assume that at initial time $t=0$ particles are~$f_\circ$-chaotic in the following $L^2$-sense:
\begin{equation}\label{eq:Liouville-init}
\limsup_{N\uparrow\infty}\int_{\D^N}\frac{|F_N^\circ|^2}{f_\circ^{\otimes N}}<\infty,
\end{equation}
for some density $f_\circ\in\mathcal P(\D)\cap L^\infty(\D)$. This assumption is the natural condition for the $L^2$-duality method developed below. It is trivially satisfied by factorized initial data $F_N^\circ=f_\circ^{\otimes N}$, but it is stronger than the usual notions of Kac or entropic chaos; see Remark~\ref{rem:initial}.

In the macroscopic limit $N\uparrow\infty$, we seek an averaged description of the system, for example describing the evolution of the phase-space density of a typical particle, as given by the first marginal
\begin{equation*}
F_{N,1}(z)\,:=\,\int_{\D^{N-1}}F_N(z,z_2,\ldots,z_N)\,dz_2\ldots dz_N.
\end{equation*}
As is well known, formally neglecting correlations between particles suggests that $F_{N,1}$ should converge to a solution $f\in L^\infty(\R^+;\mathcal P(\D)\cap L^\infty(\D))$ of the mean-field Vlasov equation
\begin{equation}\label{eq:Vlasov}
\partial_tf+v\cdot\nabla_xf+(K\ast f)\cdot\nabla_vf=\alpha \Delta_vf,\qquad f|_{t=0}=f_\circ,
\end{equation}
where we define $(K\ast f)(x):=\int_{\D}K(x-x')\,f(z')\,dz'$. More generally, for every fixed $k\ge1$, the $k$th marginal
\begin{equation}\label{eq:def-FN1}
F_{N,k}(z_1,\ldots,z_k)\,:=\,\int_{\D^{N-k}}F_N(z_1,\ldots,z_N)\,dz_{k+1}\ldots dz_N
\end{equation}
is expected to converge to the tensor product $f^{\otimes k}$. This expected property is known as propagation of chaos. For singular interactions, and especially at vanishing temperature $\alpha=0$, its rigorous justification remains a major open problem; we refer to Section~\ref{sec:literature} for a discussion of the literature.

\subsection{Main results}
In the present work, we introduce a new {\it dual hierarchical method} for the mean-field limit; a general account of the method is postponed to Section~\ref{sec:method}.
Its main feature is that it allows us to treat general locally square-integrable interaction forces, including at vanishing temperature $\alpha=0$, without any sign condition or potential structure on the interaction. This yields the first propagation-of-chaos result for deterministic second-order dynamics at this level of generality. We refer to Section~\ref{sec:gen} for extensions and applications, and to Section~\ref{sec:literature} for a detailed comparison with previous results.

In a companion paper~\cite{DuJa2} by the last two authors, we further refine the method beyond the present square-integrable framework, at the price of substantially finer estimates on the dual correlation hierarchy exploiting kinetic regularization effects. This leads in particular to the first derivation of the 2D Vlasov-Poisson equation for Coulomb particles. The same paper also characterizes empirical fluctuations around the mean-field limit, showing that the same method provides information beyond propagation of chaos.

Another important feature of the method is that it does not rely on the existence of renormalized solutions of the $N$-particle Liouville equation~\eqref{eq:Liouville}. In Appendix~\ref{sec:dual-sol}, we introduce instead a notion of \emph{weak duality solutions}, for which we prove global existence whenever $K\in L^1_\loc(\Omega;\R^d)$.
Such solutions need not be unique, as the low regularity of the interaction force does not in general ensure a renormalization property for the Liouville equation. Moreover, weak duality solutions need not even remain probability distributions, consistently with possible breakdown of the underlying particle dynamics through collisions or escape to infinity in phase space. The following nevertheless applies to every weak duality solution in this class. In post-collisional regimes, this avoids in particular the need to prescribe a microscopic selection rule for the continuation.

\begin{theorem}[Square-integrable interactions]\label{th:main}
Let
\begin{equation}\label{hypothesis-Kernel}
0\le\alpha<\infty,\qquad 
K\in L^2_\loc(\Omega;\R^d),\qquad 
\sup_{|x|\ge 1}|K(x)|<\infty.
\end{equation} 
Consider a symmetric global weak duality solution $F_N\in L^\infty_\loc(\R^+;L^1_+(\D^N))$ of the Liouville equation~\eqref{eq:Liouville}, in the sense introduced in Appendix~\ref{sec:dual-sol}, with initial data~$F_N^\circ\in\Pc(\D^N)\cap L^1(\D^N)$ satisfying the $f_\circ$-chaoticity condition~\eqref{eq:Liouville-init} for some $f_\circ\in \mathcal P(\D)\cap L^\infty(\D)$.
For some $T>0$, let 
\begin{equation}\label{hypothesis-f0}
f\in L^\infty(0,T;\mathcal P(\D)\cap L^\infty(\D))
\end{equation}
be a weak solution of the Vlasov equation~\eqref{eq:Vlasov} with initial data~$f_\circ$,
and assume that
\begin{gather}\label{hypothesis-f}
\rho_f:=\int_{\R^d}f(\cdot,v)\,dv\in L^\infty([0,T]\times\Omega),\qquad \nabla_v\sqrt f \in L^1(0,T;L^2(\D)),\\
\text{and}~~K\ast f\in L^1(0,T;H^1_\loc(\Omega))~~\text{if $\alpha=0$}.\nonumber
\end{gather} 
Then, for every $k\ge1$ and $t\in[0,T]$,
\[F_{N,k}(t)\rightharpoonup f(t)^{\otimes k}\quad\text{narrowly as finite measures on $\D^k$}.\]
\end{theorem}

\begin{remark}\label{rem:initial}\
\begin{enumerate}[(a)]
\item \emph{Possible loss of mass:}\\
Although weak duality solutions need not remain probability distributions, as mass may escape to infinity in phase space in finite time, the above result nevertheless applies and identifies the limiting behavior of all fixed marginals. In particular, since the limit $f(t)^{\otimes k}$ is a probability measure, the narrow convergence implies
\[\int_{\D^N}F_N(t)=\int_{\D^k}F_{N,k}(t)\to1\qquad\text{as $N\uparrow\infty$},\]
thus showing that any possible loss of mass becomes negligible in the mean-field limit.

\smallskip
\item \emph{Time of validity and regularity of the mean-field solution:}\\
The above result holds for as long as the mean-field solution $f$ enjoys the regularity required in~\mbox{\eqref{hypothesis-f0}--\eqref{hypothesis-f}}. In Appendix~\ref{sec:regularity}, we briefly recall that these assumptions are satisfied, at least locally in time, for sufficiently regular initial data.
Although the Fisher-information assumption in~\eqref{hypothesis-f} only involves velocity derivatives, the transport coupling naturally requires regularity of the initial data in both $x$ and $v$ to propagate the desired bounds locally in time.\\
Yet, these assumptions are not optimal in our method. More intrinsically, the main singular two-particle coefficient entering the dual hierarchy at order $O(1)$ is
\begin{equation}\label{eq:def-Kf}
K_f(z,z'):=\big(K(x-x')-K\ast f(x)\big)\cdot\nabla_v\log f(z),
\end{equation}
cf.~Lemma~\ref{propCn}, and we only need to assume $K_f\in L^1(0,T;L^2(f^{\otimes2}))$.
When $\alpha=0$, the proof further requires a renormalization property for weighted $L^2$ solutions of the linearized mean-field equation, which is ensured by the assumption $K\ast f\in L^1(0,T;H^1_\loc(\Omega))$; cf.~\cite{DiLi}. Assumptions~\eqref{hypothesis-Kernel}--\eqref{hypothesis-f} are convenient sufficient conditions ensuring these requirements.

\smallskip
\item \emph{Initial chaos assumption:}\\
Assumption~\eqref{eq:Liouville-init} is a uniform relative $L^2$ bound, or equivalently a uniform bound on the $\chi^2$-divergence of $F_N^\circ$ with respect to the product law $f_\circ^{\otimes N}$. As will become clear from the proof, this is the natural assumption for our $L^2$-duality method.
It is stronger than the usual notions of Kac or entropic chaos:
\begin{enumerate}[---]
\item In terms of the linear correlation functions $(H_{N,m}^\circ)_m$ associated with $F_N^\circ$, as introduced in Section~\ref{sec:method}, orthogonality gives for every $1\le k\le N$,
\[\qquad\int_{\Dd^k}\Big|\frac{F_{N,k}^\circ}{f_\circ^{\otimes k}}-1\Big|^2f_\circ^{\otimes k}
=\sum_{m=1}^k\binom{k}{m}\int_{\Dd^m}|H_{N,m}^\circ|^2f_\circ^{\otimes m},\]
and thus, using $\binom{k}{m}\le\frac{k}{N}\binom{N}{m}$ for $1\le m\le k\le N$,
\[\qquad\int_{\Dd^k}\Big|\frac{F_{N,k}^\circ}{f_\circ^{\otimes k}}-1\Big|^2f_\circ^{\otimes k}
\le\frac kN\sum_{m=1}^N\binom{N}{m}\int_{\Dd^m}|H_{N,m}^\circ|^2f_\circ^{\otimes m}
=\frac kN\Big(\int_{\Dd^N}\frac{|F_N^\circ|^2}{f_\circ^{\otimes N}}-1\Big).\]
Hence, assumption~\eqref{eq:Liouville-init} implies Kac chaos, with the stronger quantitative $L^2$ convergence
\[\|F_{N,k}^\circ-f_\circ^{\otimes k}\|_{L^2(1/f_\circ^{\otimes k})}\lesssim \Big(\frac kN\Big)^{\frac12},\qquad 1\le k\le N.\]
\item It also implies entropic chaos. Indeed, Jensen's inequality yields
\[H(F_N^\circ|f_\circ^{\otimes N}):=\int_{\Dd^N}F_N^\circ\log\Big(\frac{F_N^\circ}{f_\circ^{\otimes N}}\Big)\le\log\int_{\Dd^N}\frac{|F_N^\circ|^2}{f_\circ^{\otimes N}}=O(1),\]
while relative-entropy methods generally only require the weaker condition $\frac1NH(F_N^\circ|f_\circ^{\otimes N})\to0$; see e.g.~\cite{JabWan1}.
\end{enumerate}

\smallskip\item \emph{Quantitative convergence and fluctuations:}\\
The compactness argument used in this work does not provide a convergence rate. Yet, under the slightly stronger condition $K\in H^s_\loc(\Omega)$ for some $s>0$, it can be replaced by a quantitative stability analysis of the finite-$N$ dual hierarchy around its limit, which then yields
\begin{equation}\label{eq:conv-FNk}
\|F_{N,k}(t)-f(t)^{\otimes k}\|_{C_c(\D^k)^*}\le CN^{-(\frac1Ce^{-Ct})}.
\end{equation}
This weak rate, with exponent deteriorating exponentially in time, reflects the lack of a semigroup property in the hierarchical argument. We skip the argument here and refer to a previous online version of this work (see~\verb?arXiv:2402.04695v2?). More recently, a finer quantitative analysis of the dual hierarchy was developed in~\cite{KhJa}, where the optimal $O(N^{-1})$ rate is recovered under additional regularity conditions on $K$ and $f$.
\\[1mm]
In a different direction, we show in the companion paper~\cite{DuJa2} that, under the same assumptions as in Theorem~\ref{th:main}, the duality method is able to characterize fluctuations around the mean-field limit. In particular, writing $\mu_{Z_N}=\frac1N\sum_{n=1}^N\delta_{z_n}$ with $Z_N=(z_1,\ldots,z_N)$,
we obtain in~\cite{DuJa2}, for all~$s>d$ and $t\in[0,T]$,
\[\|\mu_{Z_N}-f(t)\|_{H^{-s}(\D)}=O_\Pm(N^{-1/2})\quad\text{in probability under $F_N(t)$}.\]
This expresses the optimal central-limit scaling for empirical fluctuations, although it does not by itself provide a quantitative convergence rate for the marginals.
\end{enumerate}
\end{remark}

\subsection{Extensions and further applications}\label{sec:gen}
The duality method is fairly robust and extends beyond the second-order dynamics considered above. We first describe its adaptation to first-order particle dynamics, and then discuss second-order dynamics with general phase-space interactions, including application to alignment models.
We also refer to~\cite{FengWang} for a recent application of the method to Kac's particle system.

\subsubsection{First-order dynamics}
Instead of~\eqref{eq:SDE}, consider the inertialess particle dynamics
\begin{equation}\label{eq:SDE-1st}
dX_{i,N}=\frac1{N-1}\sum_{j:j\ne i}^NK(X_{i,N}-X_{j,N})\,dt+\sqrt{2\alpha}\,dB_{i,N},\qquad1\le i\le N,
\end{equation}
with associated Liouville equation on $\Omega^N$,
\begin{equation}\label{eq:Liouville-1st}
\partial_tF_N+\frac1{N-1}\sum_{i\ne j}^N\Div_{x_i}\big(K(x_i-x_j)F_N\big)=\alpha\sum_{i=1}^N\Delta_{x_i}F_N.
\end{equation}
In this setting, substantially more is already known about the mean-field limit; see Section~\ref{sec:literature}. Most available results for singular interactions, however, exploit additional structure of the interaction, such as gradient-flow or convexity properties, antisymmetry, or suitable non-attraction conditions.
Here again, the duality method does not require any such assumption.

In this first-order setting, compressibility introduces additional regularity requirements on $\Div(K)$, while symmetrization partially compensates for the singularity of the interaction.
Indeed, the analogue of the singular two-particle coefficient~\eqref{eq:def-Kf} entering the dual hierarchy is
\[K_f(x,x'):=\frac1{f(x)}\Div_x\Big(\big(K(x-x')-K\ast f(x)\big)f(x)\Big).\]
Decomposing the interaction kernel into its even and odd parts
\[K=K_0+K_1,\qquad K_0(-x)=K_0(x),\qquad K_1(-x)=-K_1(x),\]
and symmetrizing $K_f$ in $(x,x')$, we find
\begin{multline}\label{eq:Kfsym}
K_f^\Sym(x,x')
=\Div(K_1)(x-x')+\frac12K_1(x-x')\cdot\big(\nabla\log f(x)-\nabla\log f(x')\big)\\
+\frac12K_0(x-x')\cdot\big(\nabla\log f(x)+\nabla\log f(x')\big)
-\frac12\bigg(\frac{\Div((K\ast f)f)}{f}(x)+\frac{\Div((K\ast f)f)}{f}(x')\bigg).
\end{multline}
Hence, the divergence of the even part of the interaction cancels entirely, while the odd part of the interaction appears only through $\Div(K_1)$ and through
\begin{equation}\label{eq:def-Hf}
H_f(x,x'):=K_1(x-x')\cdot\big(\nabla\log f(x)-\nabla\log f(x')\big).
\end{equation}
If $\nabla\log f$ is Lipschitz, the latter quantity is simply bounded by $|H_f(x,x')|\lesssim|x-x'||K_1(x-x')|$. This gain of a factor $|x-x'|$ at the diagonal is reminiscent of the classical symmetrization underlying weak vorticity formulations of the 2D Euler equation~\cite{Scho95,Scho96}.

For compressible interactions, some care is further needed in the formulation of the result: the $N$-particle flow may concentrate an initially absolutely continuous distribution onto lower-dimensional sets in finite time, consistently with the sticky collisions that can occur for attractive aggregation models; see e.g.~\cite{CaDiFiLaSl11,CaJaLaVa16,JaVa16}. The natural formulation therefore uses a corresponding {\it measure-valued} notion of weak duality solutions. In Appendix~\ref{app:1storder-dualsol}, we prove existence of such solutions whenever $K,\Div(K)\in L^1_\loc(\Omega)$, and we show that a lower bound $\Div(K)\ge-C$ is a sufficient condition preventing concentration. This formulation does not prescribe any microscopic selection rule at singular events: in attractive models we do not impose, for instance, a sticky-particle continuation after collisions, nor do we define the dynamics by selecting limits of smooth regularizations of the interaction. The mean-field result below applies indiscriminately to every measure-valued weak duality solution.

\begin{theorem}[First-order dynamics]\label{th:main-1st}
Let $0\le\alpha<\infty$, and decompose $K=K_0+K_1$ into its even and odd parts.
Assume for the even part,
\[K_0\in L^2_\loc(\Omega;\R^d),\qquad\Div(K_0)\in L^1_\loc(\Omega),\]
and for the odd part,
\[K_1\in L^1_\loc(\Omega;\R^d),\quad |x|K_1\in L^2_\loc(\Omega;\R^d),\quad\Div(K_1)\in L^2_\loc(\Omega),\]
as well as
\[\sup_{|x|\ge1}\Big(|K(x)|+|\Div(K)(x)|\Big)<\infty.\]
Consider a symmetric global measure-valued weak duality solution $F_N\in C_{w^*}(\R^+; \Mcal_+(\Omega^N))$ of~\eqref{eq:Liouville-1st} in the sense introduced in Appendix~\ref{app:1storder-dualsol}, with initial data $F_N^\circ\in\Pc(\Omega^N)\cap L^1(\Omega^N)$ satisfying the analogue on $\Omega$ of the $f_\circ$-chaoticity condition~\eqref{eq:Liouville-init}, for some $f_\circ\in\mathcal P(\Omega)\cap L^\infty(\Omega)$.
For some $T>0$, let
\[f\in L^\infty_\loc(\R^+;\mathcal P(\Omega)\cap L^\infty(\Omega))\]
be a weak solution of the McKean-Vlasov equation
\[\partial_tf+\Div((K\ast f)f)=\alpha\Delta f,\qquad f|_{t=0}=f_\circ,\]
and assume that it satisfies
\begin{gather*}
\nabla\sqrt f\in L^1(0,T;L^2(\Omega)),\quad H_f\in L^1(0,T;L^2(f^{\otimes2})),\\
\text{and}~~\nabla f\in L^1(0,T;L^\infty(\Omega))~~\text{if $\alpha=0$},
\end{gather*}
with $H_f$ defined in~\eqref{eq:def-Hf}. Then, for every $k\ge1$ and $t\in[0,T]$,
\[F_{N,k}(t)\rightharpoonup f(t)^{\otimes k}\quad\text{narrowly as finite measures on $\Omega^k$}.\]
\end{theorem}

More concretely, focussing on power-law behaviors, the above proves propagation of chaos, as long as the mean-field solution is regular enough, in case of interaction kernels $K$ satisfying
\[|K(x)|+|x||\Div(K)(x)|\lesssim|x|^{-\gamma},\]
in the range
\begin{equation}\label{eq:power-interact}
\gamma<\left\{\begin{array}{lll}
(\frac d2+1)\wedge d&:&\text{$K$ odd divergence-free},\\[1mm]
\frac d2&:&\text{$K$ even divergence-free},\\[1mm]
(d-1)\wedge\frac d2&:&\text{$K$ even},\\[1mm]
\frac d2-1&:&\text{$K$ odd}.
\end{array}\right.
\end{equation}
Compared with previous propagation-of-chaos results, see Section~\ref{sec:literature}, this substantially improves on~\cite{Hauray09,CaChHa14,HoSc26} for even or divergence-free interactions, without any non-attractiveness assumption and without requiring control on the full gradient $\nabla K$. For odd, non-divergence-free interactions, on the other hand, the range $\gamma<\frac d2-1$ coincides with the earlier threshold of~\cite{Hauray09,CaChHa14}, while~\cite{HoSc26} reaches a larger exponent under additional non-attractiveness assumption.

\begin{example}[Sedimentation for inertialess point particles]\label{ex:sedim}
A notable new application is the 3D point-particle model for inertialess sedimentation, namely~\eqref{eq:SDE-1st} with
\[K(x)=\Phi(x)g,\qquad\Phi(x)=\frac1{8\pi|x|}\Big(\Id+\frac{x\otimes x}{|x|^2}\Big),\]
where $g$ denotes gravity and $\Phi$ is the Oseen tensor. Since $\Phi$ is even and divergence-free, this fits in the range~\eqref{eq:power-interact}.
The mean-field limit for this point-particle model was first justified in~\cite{Hauray09,CaChHa14} for sufficiently well-prepared initial configurations, which was subsequently extended to suspensions of finite-size rigid particles in Stokes flow; see e.g.~\cite{Hofer18}. Propagation of chaos from chaotic initial configurations, however, has remained out of reach, even after the recent progress in~\cite{HoSc26,NgSe26}. To our knowledge, Theorem~\ref{th:main-1st} provides the first propagation-of-chaos result for this model.
\end{example}

\begin{remark}\
\begin{enumerate}[(a)]
\item \emph{Time of validity and regularity of the mean-field solution.}\\
As in Theorem~\ref{th:main}, the result holds for as long as the mean-field solution enjoys the required regularity.
In particular, for attractive aggregation-type dynamics, it applies up to the time at which the limiting mean-field density itself may develop macroscopic concentration. This does not exclude microscopic collisions before that time: some particles may already have collided and undergone a post-collisional continuation without affecting the regular mean-field limit.
\\
Regarding the regularity assumptions on $f$, the condition $H_f\in L^1(0,T;L^2(f^{\otimes2}))$ is substantially weaker than requiring $\nabla\log f$ to be Lipschitz. If $|x|K_1\in L^2_\loc(\Omega;\R^d)$, it suffices for instance that
\[\sup_{0<|h|\le1}\bigg(\int_\Omega\Big|\frac{\nabla\log f(x+h)-\nabla\log f(x)}{|h|}\Big|^2f(x+h)f(x)\,dx\bigg)^\frac12~\in L^1(0,T).\]

\smallskip
\item \emph{Possible loss of mass.}\\
On the whole space $\Omega=\R^d$, in case of non-odd singular kernels, the weak duality solutions constructed in Appendix~\ref{app:1storder-dualsol} need not remain probability measures, consistently with the possible escape of particles to spatial infinity in finite time. As in Theorem~\ref{th:main}, the above mean-field limit nevertheless ensures that any such loss of mass becomes negligible as $N\uparrow\infty$. When $\Omega=\T^d$, or when~$K$ is odd, the weak duality solutions can instead be chosen mass-preserving.
\end{enumerate}
\end{remark}

\subsubsection{General phase-space interactions}\label{sec:K(zz')}
Although we have focused so far on pairwise interaction forces depending only on differences of particle positions, this restriction is largely inessential to the duality method. Consider more generally
\begin{equation}\label{eq:gen-SDE}
\left\{\begin{array}{l}
\displaystyle
dX_{i,N}=V_{i,N}dt,\\
\displaystyle
dV_{i,N}=\frac1{N-1}\sum_{j:j\ne i}^NK(Z_{i,N},Z_{j,N})dt+\sqrt{2\alpha}dB_{i,N},\qquad1\le i\le N,
\end{array}\right.
\end{equation}
where $Z_{i,N}=(X_{i,N},V_{i,N})\in\Dd$ and where $K:\Dd\times\Dd\to\R^d$ is a general pairwise interaction.
The corresponding Liouville equation reads
\begin{equation}\label{eq:Liouville-general-K}
\partial_tF_N+\sum_{i=1}^N v_i\cdot\nabla_{x_i}F_N+\frac1{N-1}\sum_{i\ne j}^N\Div_{v_i}\!\big(K(z_i,z_j)F_N\big)
=\alpha\sum_{i=1}^N\Delta_{v_i}F_N,
\end{equation}
while the mean-field equation becomes
\begin{equation}\label{eq:gen-Vlasov}
\partial_tf+v\cdot\nabla_xf+\Div_v(K[f]f)=\alpha\Delta_vf,\qquad
K[f](z):=\int_\Dd K(z,z')f(z')dz'.
\end{equation}
As in the first-order setting, if $\Div_v(K)\ne0$, the $N$-particle flow may concentrate, so that measure-valued weak duality solutions need to be considered. We do not exploit phase-space symmetrization, since in general it brings no improvement for second-order models.

\begin{theorem}[Second-order dynamics with phase-space interactions]\label{th:geninter}
Let
\[0\le\alpha<\infty,\qquad K,\Div_v(K)\in L^1_\loc(\D\times\D).\]
Consider a symmetric global measure-valued weak duality solution $F_N\in C_{w^*}(\R^+; \Mcal_+(\D^N))$ of~\eqref{eq:Liouville-general-K} in the sense introduced in Appendix~\ref{sec:dual-sol-general}, with initial data $F_N^\circ\in\Pc(\D^N)\cap L^1(\D^N)$ satisfying the $f_\circ$-chaoticity condition~\eqref{eq:Liouville-init}, for some $f_\circ\in\mathcal P(\D)\cap L^\infty(\D)$.
For some $T>0$, let
\[f\in L^\infty_\loc(0,T;\mathcal P(\D)\cap L^\infty(\D))\]
be a weak solution of the Vlasov equation~\eqref{eq:gen-Vlasov}, with initial data $f_\circ$. Assume that $K$ and $f$ satisfy
\begin{gather}
K,\Div_v(K)\in L^1(0,T;L^2(f^{\otimes2})),\label{eq:reg-K-gen}\\
\big[(z,z')\mapsto\big(K(z,z')-K[f](z)\big)\cdot\nabla_v\log f(z)\big]~\in L^1(0,T;L^2(f^{\otimes2})),\nonumber\\
|K[f](t,x,v)|\le C(1+|v|),\quad\nabla_vK[f]\in L^\infty([0,T]\times\D),\nonumber\\
\text{and}~~K[f]\in L^1(0,T;H^{1}_\loc(\D))~~\text{if $\alpha=0$}.
\end{gather}
Then, for every $k\ge1$ and $t\in[0,T]$,
\[F_{N,k}(t)\rightharpoonup f(t)^{\otimes k}\quad\text{narrowly as finite measures on $\D^k$}.\]
\end{theorem}

\begin{example}[Singular Cucker-Smale alignment]\label{ex:alignment}
The Cucker-Smale model corresponds to~\eqref{eq:gen-SDE} with
\[K(z,z'):=\psi(x-x')(v'-v),\]
where $\psi:\Omega\to[0,\infty]$ is the communication weight.
In this case,
\begin{equation*}
\Div_v(K)(z,z')=-d\psi(x-x'),\qquad K[f](z)=(\psi\ast j_f)(x)-(\psi\ast\rho_f)(x)v,
\end{equation*}
in terms of $\rho_f(x):=\int_{\R^d}f(x,v)\,dv$ and $j_f(x):=\int_{\R^d}vf(x,v)\,dv$,
and the assumptions of Theorem~\ref{th:geninter} can thus be made quite explicit:
it applies whenever
\[\psi\in L^2_\loc(\Omega),\qquad
\sup_{|x|\ge1}|\psi(x)|<\infty,\]
as long as the mean-field solution $f$ satisfies, for instance,
\begin{gather}
f\in L^\infty(0,T;\Pc(\D)\cap L^\infty(\D)),
\qquad (1+|v|)\nabla_v\sqrt f\in L^1(0,T;L^2(\D)),\nonumber\\
\int_{\R^d}(1+|v|^2)f(\cdot,v)dv\in L^\infty(0,T;L^1(\Omega)\cap L^\infty(\Omega)),\nonumber\\
\text{and}~~\int_0^T\int_\D(1+|v|)\,|\nabla_xf(z)|\,dz<\infty~~\text{if $\alpha=0$}.\label{eq:reg-ass-f-CS}
\end{gather}
For the standard singular communication weight
\begin{equation}\label{eq:standard-psi}
\psi(x)=|x|^{-\beta},
\end{equation}
we therefore obtain the mean-field limit throughout the range $\beta<\frac d2$,
as long as the mean-field solution remains regular enough in the sense of~\eqref{eq:reg-ass-f-CS}. In dimensions $d>1$, this strictly enlarges the previous mean-field result of~\cite{MuPe}, which was restricted to $\beta<\frac12$.
Further singularities had previously been reached by exploiting additional structure in special variants of the model. In 1D, the mean-field limit was justified in~\cite{ChoiZhang21} for all $\beta<1$, using a change of variables that reduces the second-order dynamics to a first-order one. This reduction was extended in~\cite{PeszekPoyato} to arbitrary dimension for the specific matrix-valued weight $\psi(x)=|x|^{-\beta}(\Id-\beta\frac{x\otimes x}{|x|^2})$, for which the mean-field result was then obtained for all $\beta<1$. By contrast, the above does not exploit any such structure and applies equally to the standard scalar weight~\eqref{eq:standard-psi}. It also extends the admissible range for the matrix-valued model to $\beta<\frac d2$, which improves on~\cite{PeszekPoyato} in dimensions $d\ge3$.
\end{example}

\subsection{A new dual hierarchical method}\label{sec:method}
We give a brief account of the main ideas underlying the method.
Classical hierarchical approaches to mean-field limits start from the BBGKY hierarchy for the marginals
\begin{multline*}
\partial_tF_{N,m}+\sum_{i=1}^m(v_i\cdot\nabla_{x_i}-\alpha\Delta_{v_i})F_{N,m}+\frac1{N-1}\sum_{i\ne j}^m K(x_i-x_j)\cdot\nabla_{v_i}F_{N,m}\\[-2mm]
+\frac{N-m}{N-1}\sum_{i=1}^m\int_\Dd K(x_i-x_*)\cdot\nabla_{v_i} F_{N,m+1}(\cdot,z_*)\,dz_* =0.
\end{multline*}
For singular interactions, a first difficulty in passing to the limit in this hierarchy is that the marginals are a priori only bounded measures, whereas the last interaction term involves integration against the singular kernel~$K$. One may therefore try to propagate additional integrability of the marginals, but a second and more fundamental difficulty then appears: the interaction term also loses one derivative in velocity. At vanishing temperature $\alpha=0$, there is no regularizing mechanism available to compensate for this loss, and the usual direct iteration of the hierarchy does not provide closed estimates. This is in sharp contrast with, for instance, the mean-field limit in the quantum setting or for the Kac model, where the interaction operators are bounded on the natural spaces and direct hierarchical arguments can be closed; see e.g.~\cite{PaPuSi19}. When $\alpha>0$, velocity diffusion compensates for the loss of derivatives and the direct hierarchical strategy can also be closed, at least locally in time; see~\cite{Lacker,BrJaSo}.

The main observation underlying the present work is that this derivative loss in the hierarchy is formally compensated by the fluctuation scaling. To explain this, we introduce the linear correlation functions $(H_{N,k})_k$ associated with the $N$-particle density relative to the tensorized mean-field state~$f^{\otimes N}$: denoting by $\pi_f$ the orthogonal projection onto constants in~$L^2(f)$,
\begin{equation}\label{eq:def-proj-pif}
\pi_fh:=\int_\D hf,
\end{equation}
we define
\begin{equation}\label{eq:def-HNk}
H_{N,k}:=(\Id-\pi_f)^{\otimes k}\Big(\frac{F_{N,k}}{f^{\otimes k}}\Big).
\end{equation}
Thus, for instance, assuming that $F_N$ has unit mass,
\[H_{N,1}f=F_{N,1}-f,\qquad H_{N,2}f^{\otimes2}=F_{N,2}-F_{N,1}\otimes f-f\otimes F_{N,1}+f^{\otimes2}.\]
These quantities can equivalently be viewed as the Hoeffding components of the particle distribution~$F_N$ with respect to~$f^{\otimes N}$, and provide a natural hierarchical representation of fluctuations around the mean-field state. More precisely, $H_{N,k}$ represent the $k$-particle component of the fluctuations of the empirical measure around~$f$. The usual central-limit scaling of empirical fluctuations therefore suggests
\begin{equation}\label{eq:expect-scaling-H}
H_{N,k}=O(N^{-k/2}).
\end{equation}
Of course, such estimates are {\it not} available a priori: they already contain substantially more information than the mere mean-field limit. Yet, they allow to identify the appropriate scaling of the hierarchy: introducing the rescaled quantities
\[\bar H_{N,k}:=N^{k/2}H_{N,k},\]
the BBGKY hierarchy is turned into a system of the form
\begin{equation}\label{eq:intro-fluct-hier}
(\partial_t+L_{N,k})\bar H_{N,k} = S_{N,k}^{=}\bar H_{N,k-2} +\frac1{\sqrt N} \Big(S_{N,k}^{+}\bar H_{N,k+1} + S_{N,k}^{-}\bar H_{N,k-1} \Big).
\end{equation}
The crucial point lies in the structure of this rescaled hierarchy. The double-creation operator~$S_{N,k}^{=}$ is the only coupling that remains $O(1)$ as $N\uparrow\infty$: although it still contains the singular interaction kernel~$K$, it involves no loss of velocity derivatives. In contrast, the creation and annihilation operators~$S_{N,k}^\pm$ both lose one velocity derivative, but they appear with the small prefactor $N^{-1/2}$.

Hence, the fluctuation scaling reveals that the loss of velocity derivatives enters the hierarchy only through a singular perturbation. This new observation is the basic mechanism that makes a hierarchical approach conceivable even at vanishing temperature $\alpha=0$. The difficulty, however, is that this favorable structure would only become useful if one had a priori control at the fluctuation scale~\eqref{eq:expect-scaling-H}.

This is where duality comes into play. On the primal side, by orthogonality, estimates on the linear correlations would amount to controlling the weighted $L^2$ norm of $F_N$:
\begin{equation}\label{eq:L2weight-DN}
\sum_{k=0}^N\binom{N}{k}\int_{\Dd^k}|H_{N,k}|^2f^{\otimes k}=\int_{\Dd^N}\frac{|F_N|^2}{f^{\otimes N}},
\end{equation}
but no useful propagation estimate is available for the latter in the regime considered here; $F_N$ is only naturally bounded in $L^1$. Instead, we appeal to duality and consider the backward Liouville evolution of \emph{bounded} observables $\Phi_N$, which are formally in duality with $F_N$ in the sense that
\[\frac{d}{dt}\int_{\D^N} \Phi_NF_N=0.\]
On this dual side, the weighted $L^2$ norm is now controlled by
\[\int_{\Dd^N}|\Phi_N|^2f^{\otimes N} \le \|\Phi_N\|_{L^\infty(\Dd^N)}^2,\]
while an $L^\infty$-bound follows from the maximum principle for the backward dual Liouville equation. By orthogonality, as in~\eqref{eq:L2weight-DN}, decomposing $\Phi_N$ into its linear correlations therefore yields, essentially for free, the a priori estimates that were missing on the primal side; see Lemma~\ref{lem:apriori-sym}. We then consider the corresponding hierarchy of equations for these \emph{dual} correlations at the fluctuation scale and use the available a priori bounds to show that all interaction terms with derivative loss are perturbative and vanish as $N\uparrow\infty$; see Lemma~\ref{propbarCn}. The limit hierarchy is thus free of these singular perturbative terms and satisfies a uniqueness principle; see Lemma~\ref{lem:unique-hier}. This allows to propagate the triviality of the limiting dual correlations backwards in time, which ultimately yields propagation of chaos for the original particle system.

The method thus combines two ingredients with distinct precedents in the literature: correlation expansions on the one hand, and duality arguments on the other.
Correlation functions and their evolution hierarchies have a long history in kinetic theory. Nonlinear cumulant expansions have been extensively used both for mean-field and hard-sphere dynamics; see for instance~\cite{GeGa12,LuMa16,LuMaNo18,BoGaSaSi1,BoGaSaSi2,Duer,DuJa4}. The linear correlation functions~\eqref{eq:def-HNk} considered here are instead closely related to the ``correlation errors'' introduced in~\cite{PuSi17,PaPuSi19}. The resulting linear decomposition of the $N$-particle density has proved particularly useful near equilibrium; see notably~\cite{BGSR-17}, and more recently~\cite{DuerSain,DuJa,DuJa3} for mean-field setting.
In those works, the equilibrium or near-equilibrium structure provides a priori control of the weighted $L^2$ norm~\eqref{eq:L2weight-DN}, and hence of the linear correlations. The situation considered here is different: away from equilibrium, no comparable estimate is available for the primal hierarchy, and passing to the dual equation is precisely what restores such control through the boundedness of the observables.

Duality arguments themselves are classical in the theory of transport and continuity equations. They play a central role in the DiPerna-Lions theory~\cite{DiLi} and were further developed, for instance, in~\cite{CriSpi,CriNobSpi,CiaCriSpi}.
The general principle relevant here is that a rough forward evolution may be studied through its adjoint equation when the latter enjoys stronger a priori estimates. To our knowledge, however, such ideas had never been used at the level of the Liouville equation in the context of mean-field limits.

\subsection{Comparison with previous results}\label{sec:literature}
For Lipschitz interaction forces, the mean-field limit goes back to the classical works~\cite{BraHep,Dob,Sznitman}. Derivative assumptions on the force were later removed in~\cite{JabWan1}, where arbitrary bounded interaction forces were treated by a relative-entropy method at the level of the Liouville equation. For unbounded forces, the situation is substantially more delicate and strongly depends on the structure and the order of the dynamics, the presence of diffusion, and on additional structural properties of the interaction.

\subsubsection{Second-order dynamics}
For genuinely second-order systems, few results are available for unbounded interactions.
In dimensions $d\ge2$, an approach based on Wasserstein stability together with quantitative control of the minimal interparticle distance was developed in~\cite{HauJab1,HauJab2}. At vanishing temperature, for well-prepared initial configurations, it yields the mean-field limit for interactions satisfying
\[|K(x)|+|x||\nabla K(x)|\lesssim |x|^{-\gamma},\]
in the range $\gamma<1$. For generic iid initial configurations, the corresponding propagation-of-chaos result is obtained in~\cite{HauJab2} only for $d\ge3$. In 3D, it appears possible to extend this range to some $\gamma>1$, possibly up to $\gamma<4/3$, following~\cite{Grass19}. For power-law interactions, the square-integrable framework of the present work reaches the range $\gamma<\frac d2$ for all $d\ge1$. It thus improves previous thresholds for all $d\ge3$, while in 2D it recovers the range $\gamma<1$ without requiring any assumption on $\nabla K$. It also applies in 1D, which is not covered by previous results. The present work still does not reach the Coulomb force in dimensions~$d\ge2$; the 2D Coulomb case is covered by the refined duality method in the companion paper~\cite{DuJa2}.

At positive temperature, velocity diffusion can be used to compensate for the derivative loss in the BBGKY hierarchy and to close direct hierarchical methods. Using relative entropy, Lacker~\cite{Lacker} obtained in this way optimal convergence rates for marginals under an exponential integrability condition on the interaction. A different hierarchical approach was developed in~\cite{BrJaSo} for repulsive interactions deriving from a potential, $K=-\nabla W$, under suitable integrability assumptions on $K$ and $W$. Although restricted to short times, this yields the first derivation of the 2D Vlasov-Poisson-Fokker-Planck equation together with a partial 3D result, which has recently been completed in~\cite{FengWang2}. These positive-temperature results reach substantially more singular interactions than Theorem~\ref{th:main}, but the latter applies equally at vanishing temperature and for any time interval on which the required mean-field regularity holds, without structural assumptions on~$K$.

More recently, Nam and Nguyen~\cite{NamNguyen26} extended our duality method beyond the $L^2$ framework using Morawetz-type estimates, for particle systems with vanishing but strictly positive velocity diffusion. Yet, their result does not cover the deterministic case $\alpha=0$.

For singular alignment dynamics, dispersion and alignment structures have also been exploited to obtain mean-field limits for specific classes of singular communication kernels; see e.g.~\cite{MuPe,ChoiZhang21,PeszekPoyato}. The present work enlarges the range of admissible singular interactions without relying on such model-specific structures; see Example~\ref{ex:alignment} for a detailed comparison.

\subsubsection{First-order dynamics}
For first-order systems~\eqref{eq:SDE-1st}, the absence of inertia makes singular interactions substantially more tractable,
and a much larger literature is available.
Classical results at vanishing temperature include the derivation of the 2D Euler equation from point-vortex dynamics~\cite{Goodman91,GooHouLow90,Scho95,Scho96}. At positive temperature, we refer in particular to~\cite{FHM-JEMS,Osada,JabWan2} for propagation of chaos toward the 2D Navier-Stokes equation.

The Wasserstein/minimal-distance strategy of~\cite{HauJab1,HauJab2} was adapted to first-order dynamics in~\cite{Hauray09,CaChHa14}. For suitably well-prepared initial configurations, it gives the mean-field limit for interactions satisfying
\[|K(x)|+|x||\nabla K(x)|\lesssim|x|^{-\gamma},\]
throughout the sub-Coulomb range $\gamma<d-1$.
Such microscopic preparation is, however, not typically satisfied by iid samples, which constitute the canonical setting for propagation of chaos. For iid initial configurations, the admissible range is more restrictive, namely $\gamma<\frac{d}{2}-1$. More recently, under a non-attractiveness assumption, this was improved in~\cite{HoSc26} to $\gamma<\frac{d-1}{2}$ for $d\ge3$ and $\gamma<\frac13$ for $d=2$.
Theorem~\ref{th:main-1st} is complementary to these results: it requires neither non-attractiveness nor control of the full gradient $\nabla K$, and yields substantially improved ranges for even or divergence-free interactions; see~\eqref{eq:power-interact} and Example~\ref{ex:sedim} for a detailed comparison.

A different family of approaches exploits variational or energetic structure.
At vanishing temperature, the modulated-energy method yields quantitative mean-field results for repulsive Coulomb and Riesz interactions~\cite{Du16,Se,NgRoSe}. It has been extended to positive temperature in the form of a modulated {\it free-energy} method~\cite{BrJaWa,BreJabWan}; see also~\cite{RoSe} for global-in-time estimates. We refer further to~\cite{FoTa,ChRoSe2} for singular attractive interactions, and to~\cite{GuiLebMon,RoSe,ChRoSe1} for uniform-in-time propagation-of-chaos results in several settings.
The scope of Theorem~\ref{th:main-1st} is complementary: it does not reach the full Riesz range accessible through modulated-energy methods, but, within its square-integrable regime, it requires neither repulsiveness nor a potential or gradient-flow structure, and allows general interactions with both odd and even components.

More recently, after the completion of the present work, a so-called multiscale mollification method was introduced in~\cite{NgSe26} for a broad class of first-order interactions that may be attractive and need not derive from a potential. Their assumptions distinguish between the odd and even parts of the interaction: roughly speaking, they allow singularities $|x|^{-\gamma}$ with $\gamma<d-1$ for the odd part, but only $\gamma<d-2$ for the even part, and they further reach Coulomb interactions for $d\le2$. 
Notably, in dimensions $d\le 3$, these ranges are more restrictive than those of Theorem~\ref{th:main-1st} for even or divergence-free interactions, cf.~\eqref{eq:power-interact}.
For odd, non-divergence-free interactions, in contrast, their admissible range is substantially larger.
The assumptions on the initial data are, however, much more restrictive: their results require well-prepared initial configurations and exclude the canonical iid setting. In particular, their results do not directly yield propagation of chaos from iid initial data.

\subsubsection{Other regimes}

A large complementary literature considers singular forces regularized at an $N$-dependent microscopic scale. For interactions behaving as $|K(x)|\lesssim|x|^{-\gamma}$, one replaces $K$ by a regularization $K_N$ at a scale $\e_N\downarrow0$. This goes back, among others, to~\cite{GanVic,GanVic2,VA,Wollman}, and substantially smaller cutoff scales have been reached more recently in~\cite{Laz,LazPic,FePi,FePi1,FePi2}. In particular, recent work on Coulomb interactions reaches a cutoff of order $\e_N=N^{-5/12+\kappa}$ in 3D~\cite{FePi1}, for arbitrary $\kappa>0$, and $\e_N=N^{-2}$ in 2D~\cite{FePi2}. Such results approximate Coulomb dynamics at remarkably small microscopic scales, but retain a nonzero cutoff for every finite $N$.

In the perturbative regime near Gibbs equilibrium, the linearized mean-field equation was derived in~\cite{DuJa3} from second-order particle dynamics for a broad class of forces deriving from a potential, including Coulomb interactions in dimensions $d\le3$. The present work instead concerns the fully nonlinear mean-field limit away from equilibrium.

\subsection{Organization of the paper}
In Section~\ref{sec:dual-ref}, we reformulate the mean-field limit in terms of bounded solutions of a backward dual Liouville equation. Dual correlations are then introduced in Section~\ref{sec:correl}, where we also establish the key $L^2$ a priori estimates. The argument then proceeds to a weak compactness argument, and limit correlations are characterized in Section~\ref{sec:dual-hier} as the solutions of a dual hierarchy, whose uniqueness concludes the proof of Theorem~\ref{th:main}.
Extensions to first-order dynamics and general phase-space interactions, including the proofs of Theorems~\ref{th:main-1st} and~\ref{th:geninter}, are discussed in Section~\ref{sec:extensions}.

We postpone to Appendix~\ref{sec:dual-sol} the definition and basic properties of weak duality solutions of the Liouville equation. Appendix~\ref{sec:app} records the full finite-$N$ hierarchy for dual correlations, which is not needed for the argument but is a useful tool for refined estimates. Finally, Appendix~\ref{sec:regularity} briefly recalls standard local regularity estimates for the mean-field equation.

\section{Dual reformulation for mean-field limits}\label{sec:dual-ref}
Henceforth, we let the space domain $\Omega$ be either the whole space $\R^d$ or the periodic torus $\T^d$ in dimension $d\ge1$, we fix $0\le\alpha<\infty$, an interaction kernel $K$ with
\begin{equation}\label{eq:K-ass-0}
K\in L^1_\loc(\Omega;\R^d),\qquad \sup_{|x|\ge1}|K(x)|<\infty,
\end{equation}
and we fix an initial density $f_\circ\in\mathcal P(\D)\cap L^\infty(\D)$ on the phase space $\D:=\Omega\times\R^d$. We also let~$f\in L^\infty_\loc(\R^+;\mathcal P(\D)\cap L^\infty(\D))$ be a global weak solution of the Vlasov equation~\eqref{eq:Vlasov} with initial data~$f_\circ$.

Our approach starts from a simple dual reformulation of the mean-field limit. We first record the exact identity underlying the method; see~\eqref{eq:duality-FNPhiN} below. In the case of factorized initial data $F_N^\circ=f_\circ^{\otimes N}$, the initial error term in this identity vanishes, and the mean-field limit is then reduced to the control of a single dynamical term involving the solution $\Phi_N$ of a backward dual Liouville equation.
This reformulation will prove particularly convenient as the dual evolution acts on the space of bounded observables, for which much better a priori estimates are available.
The notion of weak duality solutions introduced in Appendix~\ref{sec:dual-sol} is precisely designed to ensure that this identity~\eqref{eq:duality-FNPhiN} remains valid for singular interactions. Beyond this duality relation, the rest of the method holds for an arbitrary bounded weak solution of the dual Liouville equation.

\begin{lemma}
Let $F_N\in L^\infty_\loc(\R^+;L^1(\D^N))$ be a symmetric global weak duality solution of the Liouville equation~\eqref{eq:Liouville}, in the sense of Appendix~\ref{sec:dual-sol}, with initial data $F_N^\circ\in L^1(\D^N)$.
Given $k\ge1$, $T>0$, and $h\in C^\infty_c(\D)$, let $\Phi_{N}\in L^\infty((-\infty,T]\times\D^N)$ be a symmetric bounded weak solution of the backward dual Liouville equation on $(-\infty,T]$,
\begin{equation}\label{eq:PhiN}
\partial_t\Phi_{N}+\sum_{i=1}^N\Big(v_i\cdot\nabla_{x_i}\Phi_N+\frac1{N-1}\sum_{j:j\ne i}^NK(x_i-x_j)\cdot\nabla_{v_i}\Phi_{N}\Big)
=-\alpha\sum_{i=1}^N\Delta_{v_i}\Phi_N,
\end{equation}
with final condition
\begin{equation}\label{eq:PhiN-init}
\Phi_N|_{t=T}=\Phi_N^T,  \qquad\Phi_{N}^T(z_1,\ldots,z_N):=\binom{N}{k}^{-1}\sum_{1\le i_1<\ldots<i_k\le N}h(z_{i_1})\ldots h(z_{i_k}),
\end{equation}
such that $F_N$ and $\Phi_N$ are in duality in the sense of Appendix~\ref{sec:dual-sol}.
Further assume that the mean-field solution $f$ satisfies
\begin{equation}\label{eq:freg-00}
\rho_f\in L^\infty([0,T]\times\Omega),\qquad\nabla_v f\in L^1([0,T]\times\D).
\end{equation}
Then we have
\begin{equation}\label{eq:duality-FNPhiN}
\int_{\D^k}h^{\otimes k}\big(F_{N,k}(T)-f(T)^{\otimes k}\big)
=\int_{\D^N}\Phi_N(0)\big(F_N^\circ-f_\circ^{\otimes N}\big)
-N\int_0^T\Big(\int_{\D^N}K_f(z_1,z_2)\Phi_N f^{\otimes N}\Big)dt,
\end{equation}
in terms of
\begin{equation}\label{eq:def-Vf}
K_f(z,z'):=\big(K(x-x')-K\ast f(x)\big)\cdot\nabla_{v}\log f(z)~~\in L^1(0,T;L^1(f^{\otimes2})).
\end{equation}
\end{lemma}

\begin{proof}
As $F_N$ and $\Phi_N$ are in duality in the sense of Appendix~\ref{sec:dual-sol}, we have
\[\int_{\D^N}\Phi_N^TF_N(T)\,=\,\int_{\D^N}\Phi_N(0)F_N^\circ.\]
By the choice~\eqref{eq:PhiN-init} of the final condition $\Phi_N^T$ and by symmetry of $F_N$, this gives
\begin{equation*}
\int_{\D^k}h^{\otimes k}F_{N,k}(T)
\,=\,\int_{\D^N}\Phi_N(0)F_N^\circ
\end{equation*}
As the choice~\eqref{eq:PhiN-init} also yields $\int_{\D^k}h^{\otimes k}f(T)^{\otimes k}=\int_{\D^N}\Phi_N(T)f(T)^{\otimes N}$, we get
\begin{eqnarray*}
\int_{\D^k}h^{\otimes k}\big(F_{N,k}(T)-f(T)^{\otimes k}\big)
&=&\int_{\D^N}\Phi_N(0)F_N^\circ-\int_{\D^N}\Phi_N^T\,f(T)^{\otimes N}\\
&=&\int_{\D^N}\Phi_N(0)\big(F_N^\circ-f_\circ^{\otimes N}\big)-\int_0^T\Big(\frac{d}{dt}\int_{\D^N}\Phi_N(t)f(t)^{\otimes N}\Big)\,dt.
\end{eqnarray*}
Comparing the equations for $\Phi_N$ and for $f^{\otimes N}$, the last term can be computed as follows,
\begin{equation}\label{eq:compare-N-lim-pde}
\int_0^T\Big(\frac{d}{dt}\int_{\D^N}\Phi_N(t)f(t)^{\otimes N}\Big)\,dt
=\frac1{N-1}\sum_{i\ne j}^N\int_0^T\Big(\int_{\D^N}K_f(z_i,z_j)\Phi_Nf^{\otimes N}\,\Big)dt,
\end{equation}
where $K_f$ is defined in the statement.
This identity is readily justified by a standard approximation argument for the bounded weak Liouville solution $\Phi_N$ and the weak Vlasov solution $f$: after combining the two equations, the common transport and diffusion terms cancel, and the remaining term is integrable provided that
\[K_f\in L^1(0,T;L^1(f^{\otimes2})).\]
This condition is ensured in particular by the regularity assumptions~\eqref{eq:freg-00} on $f$, together with~\eqref{eq:K-ass-0}. Recalling that~$\Phi_N$ is a symmetric function in its $N$ variables, the conclusion~\eqref{eq:duality-FNPhiN} follows.
\end{proof}

For general initial data satisfying the $L^2$ chaos assumption~\eqref{eq:Liouville-init}, the initial error term in identity~\eqref{eq:duality-FNPhiN} need not be small for the choice of final observable~\eqref{eq:PhiN-init}. We circumvent this difficulty by considering instead observables associated with the characteristic functional of the empirical measure. Their final data are almost deterministic in the sense of~\eqref{eq:conv-PhiNT-ass}, and we show that this property can be propagated backwards to time zero. This leads us to the following reduction of the mean-field limit to the control of a single dynamical term, cf.~\eqref{eq:conv-CN2-0}, for a suitable class of dual observables.

\begin{proposition}\label{lem:duality}
Let~$F_N\in L^\infty_\loc(\R^+;L^1_+(\D^N))$ be a symmetric global weak duality solution of the Liouville equation~\eqref{eq:Liouville}, in the sense of Appendix~\ref{sec:dual-sol}, with initial data $F_N^\circ\in\Pc(\D^N)\cap L^1(\D^N)$ satisfying the $f_\circ$-chaoticity assumption~\eqref{eq:Liouville-init}.
Further assume that, for some $T_0>0$, the mean-field solution $f$ satisfies~\eqref{eq:freg-00} on $[0,T_0]$.
Given $T\in[0,T_0]$ and symmetric $\Phi_N^T\in L^\infty(\Dd^N)$ with
\begin{equation}\label{eq:conv-PhiNT-ass}
\|\Phi_N^T\|_{L^\infty(\Dd^N)}\le1,\qquad \lim_{N\uparrow\infty}\|\Phi_N^T-1\|_{L^2(f(T)^{\otimes N})}=0,
\end{equation}
let $\Phi_N\in L^\infty((-\infty,T]\times\D^N)$ be a symmetric global bounded weak solution of the backward dual Liouville equation~\eqref{eq:PhiN} with final condition $\Phi_N|_{t=T}=\Phi_N^T$, in the sense of Appendix~\ref{sec:dual-sol}.
Assume that for any such $\Phi_N$ we have for all $t\in[0,T]$,
\begin{equation}\label{eq:conv-CN2-0}
N\int_t^T\Big(\int_{\D^N}K_f(z_1,z_2) \Phi_N f^{\otimes N}\Big)\,\to\,0\qquad\text{as $N\uparrow\infty$},
\end{equation}
where we recall that $K_f$ is defined in~\eqref{eq:def-Vf}.
Then, for every $k\ge1$ and $t\in[0,T_0]$,
\begin{equation}\label{eq:concl-MFL}
F_{N,k}(t)\,\rightharpoonup\,f(t)^{\otimes k}\quad\text{narrowly as finite measures on $\D^k$.}
\end{equation}
\end{proposition}

\begin{proof}
Fix $T\in[0,T_0]$, $h\in C_b^\infty(\Dd)$, $\lambda\in\R$, and consider
\begin{equation}\label{eq:cosine-terminal}
\Phi_N^T(z_1,\ldots,z_N) :=\exp\bigg(\frac{i\lambda}N\sum_{j=1}^N\Big(h(z_j)-\int_\Dd hf(T)\Big)\bigg).
\end{equation}
A direct calculation yields
\begin{eqnarray}
\|\Phi_N^T-1\|_{L^2(f(T)^{\otimes N})}
&\le&\bigg\|\frac{\lambda}N\sum_{j=1}^N\Big(h(z_j)-\int_\Dd hf(T)\Big)\bigg\|_{L^2(f(T)^{\otimes N})}\nonumber\\
&=&\frac{|\lambda|}{\sqrt N}\bigg(\int_\Dd\Big(h-\int_\Dd hf(T)\Big)^2f(T)\bigg)^\frac12~\to~0,\label{eq:conv-PhiNT-ex}
\end{eqnarray}
that is, \eqref{eq:conv-PhiNT-ass}.
Now consider a symmetric global bounded weak solution $\Phi_N\in L^\infty((-\infty,T]\times\D^N)$ of the backward Liouville equation~\eqref{eq:PhiN} with final condition $\Phi_N|_{t=T}=\Phi_N^T$.
Comparing the equations for~$\Phi_N$ and~$f^{\otimes N}$, as in~\eqref{eq:compare-N-lim-pde}, we find for all $0\le t\le T$,
\[\int_{\Dd^N}\Phi_N^Tf(T)^{\otimes N}-\int_{\Dd^N}\Phi_N(t)f(t)^{\otimes N}=\frac1{N-1}\sum_{i\ne j}^N\int_t^T\Big(\int_{\Dd^N}K_f(z_i,z_j)\Phi_Nf^{\otimes N}\Big).\]
By symmetry of $\Phi_N$, assumption~\eqref{eq:conv-CN2-0} precisely ensures that this tends to $0$ as $N\uparrow\infty$. Combining this with~\eqref{eq:conv-PhiNT-ex} then yields for all $0\le t\le T$,
\begin{equation}\label{eq:conv-PhiNint}
\int_{\D^N}\Phi_N(t)f(t)^{\otimes N}\to1.
\end{equation}
Now recall that the definition of bounded weak solutions of the dual Liouville equation in Appendix~\ref{sec:dual-sol} ensures the conservation property
\begin{equation*}
\|\Phi_N(t)\|_{L^\infty(\Dd^N)}\le\|\Phi_N^T\|_{L^\infty(\Dd^N)}=1,\qquad 0\le t\le T.
\end{equation*}
Combined with~\eqref{eq:conv-PhiNint}, this implies the strong convergence
\begin{eqnarray*}
\|\Phi_N(t)-1\|_{L^2(f(t)^{\otimes N})}^2&=&\|\Phi_N(t)\|_{L^2(f(t)^{\otimes N})}^2+1-2\Re\int_{\D^N}\Phi_N(t)f(t)^{\otimes N}\\
&\le&2\Re\Big(1-\int_{\D^N}\Phi_N(t)f(t)^{\otimes N}\Big)\to0.
\end{eqnarray*}
In particular, at $t=0$, by the $f_\circ$-chaoticity assumption~\eqref{eq:Liouville-init} with $F_N^\circ\in\Pc(\D^N)$,
\begin{equation}\label{eq:conv-dual}
\Big|\int_{\Dd^N}\Phi_N(0)F_N^\circ-1\Big|\le\|\Phi_N(0)-1\|_{L^2(f_\circ^{\otimes N})}\Big(\int_{\Dd^N}\frac{|F_N^\circ|^2}{f_\circ^{\otimes N}}\Big)^\frac12\to0.
\end{equation}
Since $F_N$ is a weak duality solution of the Liouville equation in the sense of Appendix~\ref{sec:dual-sol}, we can choose the above $\Phi_N$ to be in duality with $F_N$,
\begin{equation}\label{eq:duality}
\int_{\Dd^N}\Phi_N^TF_N(T)=\int_{\Dd^N}\Phi_N(0)F_N^\circ,
\end{equation}
so that~\eqref{eq:conv-dual} entails
\[\int_{\Dd^N}\Phi_N^TF_N(T)\to1.\]
By the choice~\eqref{eq:cosine-terminal} of $\Phi_N^T$, this means that, for any $h\in C_b^\infty(\Dd)$ and $\lambda\in\R$,
\[\int_{\Dd^N}e^{i\lambda X_N^h}F_N(T)\to1,\qquad
X_N^h:=\frac1N\sum_{j=1}^Nh(z_j)-\int_\Dd hf(T).\]
Taking $\lambda=0$ shows in particular
\begin{equation}\label{eq:conv-mass}
\int_{\Dd^N}F_N(T)\to1.
\end{equation}
Hence, by nonnegativity of $F_N$, we can apply L\'evy's continuity theorem and deduce $X_N^h\to0$ in law under the probability distribution $F_N(T)/\int_{\Dd^N}F_N(T)$.
Since $|X_N^h|\le 2\|h\|_{L^\infty(\D)}$, it follows that, for all~$k\ge1$ and $h\in C^\infty_b(\D)$,
\[\int_{\Dd^N}|X_N^h|^kF_N(T)\to0.\]
Consequently, by H\"older's inequality, for all $k\ge1$ and $h_1,\ldots,h_k\in C_b^\infty(\Dd)$,
\[\int_{\Dd^N}\bigg(\prod_{\ell=1}^kX_N^{h_\ell}\bigg)F_N(T)\to0.\]
Writing $\frac1N\sum_{j=1}^Nh(z_j)=X_N^h+\int_\D hf(T)$ and expanding, we deduce
\[\int_{\Dd^N}\bigg(\prod_{\ell=1}^k\frac1N\sum_{j=1}^Nh_\ell(z_j)\bigg)F_N(T)\to\prod_{\ell=1}^k\int_\D h_\ell f(T).\]
Expanding the sums and using symmetry to rewrite the left-hand side in terms of marginals, noting that the terms with repeated indices only yield $O(N^{-1})$ errors, we conclude
\[\int_{\Dd^k}\bigg(\prod_{\ell=1}^kh_\ell(z_\ell)\bigg)F_{N,k}(T)
\to\prod_{\ell=1}^k\int_\Dd h_\ell f(T).\]
Since $h_1,\ldots,h_k\in C_b^\infty(\Dd)$ are arbitrary, since $F_{N,k}$ is nonnegative and has converging mass, cf.~\eqref{eq:conv-mass}, and since $f(T)$ is a probability measure, this implies the narrow convergence $F_{N,k}(T)\rightharpoonup f(T)^{\otimes k}$ in the space of finite measures on $\D^k$.
\end{proof}

\section{Dual correlations}\label{sec:correl}
Henceforth, let $F_N\in L^\infty_\loc(\R^+;L^1_+(\D^N))$ be a symmetric weak duality solution of the Liouville equation~\eqref{eq:Liouville}, in the sense of Appendix~\ref{sec:dual-sol}, with initial data $F_N^\circ\in \Pc(\D^N)\cap L^1(\D^N)$ satisfying the $f_\circ$-chaoticity condition~\eqref{eq:Liouville-init}. Let $\Phi_N\in L^\infty((-\infty,T]\times\D^N)$ be a symmetric bounded weak solution of the backward Liouville equation~\eqref{eq:PhiN} on $(-\infty,T]$, with final condition $\Phi_N^T$ satisfying~\eqref{eq:conv-PhiNT-ass}.
By Proposition~\ref{lem:duality}, it remains to prove, for all $t\in[0,T]$,
\begin{equation}\label{eq:red-dual}
N\int_t^T\Big(\int_{\D^N}K_f(z_1,z_2) \,\Phi_N\, f^{\otimes N}\Big)\,\to\,0\qquad\text{as $N\uparrow\infty$},
\end{equation}
As motivated in Section~\ref{sec:method}, we analyze this quantity by decomposing $\Phi_N$ into its Hoeffding components, which we refer to as {\it dual correlations}.

We first define the marginals $\{M_{N,n}\}_{0\le n\le N}$ of~$\Phi_N$ with respect to the time-dependent mean-field weight $f^{\otimes N}$ by
\begin{equation}\label{eq:marginals-MNn}
M_{N,n}(z_1,\ldots,z_n)
\,:=\,\int_{\D^{N-n}}\Phi_N(z_1,\ldots,z_N)\,f^{\otimes N-n}(z_{n+1},\ldots,z_N)\,dz_{n+1}\ldots dz_N.
\end{equation}
As $\Phi_N$ is bounded and symmetric, each marginal~$M_{N,n}$ is also bounded and symmetric in its $n$ variables. For $n=0$, notice that the $0$th marginal $M_{N,0}=\int_{\D^N}\Phi_Nf^{\otimes N}$ is not fixed, in contrast with the corresponding $0$th-order quantity on the primal side.
The dual correlations $\{C_{N,n}\}_{0\le n\le N}$ are then defined by
\begin{equation}\label{eq:defin-CNn}
C_{N,n}\,:=\,(\Id-\pi_f)^{\otimes n}M_{N,n},
\end{equation}
where $\pi_f$ denotes the orthogonal projection onto constants in the (time-dependent) weighted Hilbert space $L^2(f)$, cf.~\eqref{eq:def-proj-pif}.
By construction, each $C_{N,n}$ is bounded and symmetric in its $n$ variables, and satisfies 
\begin{equation}\label{eq:orthogonality}
\int_{\D}C_{N,n}(z_1,\ldots,z_n)\,f(z_j)\,dz_j=0\qquad\text{for all $1\le j\le n$.}
\end{equation}
Equivalently, $C_{N,n}$ is the orthogonal projection of $M_{N,n}$ in $L^2(f^{\otimes n})$ onto
\begin{equation}\label{eq:def-Hn}
\Hc_n\,:=\,\Big\{g\in L^2(f^{\otimes n})\,:\,\int_\D g(z_1,\ldots,z_n)\,f(z_j)\,dz_j=0~~\text{for all $1\le j\le n$}\Big\}.
\end{equation}

Expanding the product $(\Id-\pi_f)^{\otimes n}$ in~\eqref{eq:defin-CNn} and using that the definition~\eqref{eq:marginals-MNn} of marginals is equivalently written as $M_{N,n}=(\Id^{\otimes n}\otimes\pi_f^{\otimes N-n})\Phi_N$, we find that correlations can be expressed as linear combinations of marginals, in form of the inclusion-exclusion formula
\begin{equation}\label{eq:def-CNn-M}
C_{N,n}(z_1,\ldots,z_n)\,=\,\sum_{k=0}^n(-1)^{n-k}\sum_{\sigma\in P_k^n}M_{N,k}(z_\sigma),\qquad 0\le n\le N,
\end{equation}
where $P_k^n$ denotes the set of all subsets of $[n]:=\{1,\ldots,n\}$ with~$k$ elements, and where for an index subset $\sigma=\{i_1,\ldots,i_k\}$ we use the short-hand notation $z_\sigma:=(z_{i_1},\ldots,z_{i_k})$.
Conversely, writing $M_{N,n}=(\Id-\pi_f+\pi_f)^{\otimes n}M_{N,n}$ and expanding the product, we find that marginals can be expressed in terms of correlations, in form of the cluster expansion
\begin{equation}\label{eq:clusterexpPhin}
M_{N,n}(z_1,\ldots,z_n)\,=\,\sum_{k=0}^n\sum_{\sigma\in P_k^n}C_{N,k}(z_\sigma),\qquad 0\le n\le N.
\end{equation}

Let us now rewrite the target~\eqref{eq:red-dual} in terms of these dual correlations. By definition~\eqref{eq:def-Vf}, $K_f$ satisfies the following cancellation properties
\begin{equation}\label{eq:cancel-Vf}
\int_{\D} K_f(z_1,z_2)\,f(z_j)\,dz_j=0\qquad\text{for~$j=1,2$}.
\end{equation}
Hence, using first the definition of $M_{N,2}$ and then the cluster expansion~\eqref{eq:clusterexpPhin}, we find
\begin{equation}\label{eq:rewr-red}
N\!\int_t^T\!\Big(\int_{\D^N}K_f(z_1,z_2)\,\Phi_N\,f^{\otimes N}\Big)
=N\!\int_t^T\!\Big(\int_{\D^2}K_f \,C_{N,2}\,f^{\otimes2}\Big).
\end{equation}
In particular, since the assumptions~\eqref{hypothesis-Kernel}--\eqref{hypothesis-f} on $K$ and $f$ precisely ensure
\[K_f\in L^1(0,T;L^2(f^{\otimes2})),\]
it is enough to prove
\begin{equation}\label{eq:conv-CN2-weakstar}
NC_{N,2}\overset*\rightharpoonup0\qquad\text{in $L^\infty(0,T;L^2(f^{\otimes2}))$}.
\end{equation}
Indeed, this convergence would imply that the right-hand side of~\eqref{eq:rewr-red} tends to zero for all $t\in[0,T]$. By Proposition~\ref{lem:duality}, this yields in turn the desired mean-field limit.

The main advantage of this reformulation is that dual correlations enjoy a priori estimates directly from orthogonality and from the $L^\infty$ maximum principle for the backward Liouville equation.
As emphasized in Section~\ref{sec:method}, such estimates are not available on the primal side.

\begin{lemma}[Orthogonality and a priori bounds]\label{lem:apriori-sym}
For all $0\le n\le N$, we have
\begin{equation}\label{eq:apriori-CN}
\sup_{[0,T]}
\|C_{N,n}\|_{L^2(f^{\otimes n})}\le\binom Nn^{-\frac12}.
\end{equation}
Moreover, at the final time $T$,
\begin{equation}\label{eq:final-CN}
\binom{N}{n}^\frac12C_{N,n}(T)\to\mathds1_{n=0}\qquad\text{in $L^2(f(T)^{\otimes n})$}.
\end{equation}
\end{lemma}

\begin{proof}
Taking $n=N$ in the cluster expansion~\eqref{eq:clusterexpPhin} gives
\[\Phi_N=\sum_{n=0}^N\sum_{\sigma\in P_n^N}C_{N,n}(z_\sigma).\]
By~\eqref{eq:orthogonality}, two terms in this expansion are orthogonal in $L^2(f^{\otimes N})$ unless they involve the same index set. Symmetry then yields
\[\int_{\Dd^N}|\Phi_N|^2f^{\otimes N}=\sum_{n=0}^N\binom{N}{n}\int_{\D^n}|C_{N,n}|^2f^{\otimes n}.\]
The left-hand side is bounded above by $\|\Phi_N\|_{L^\infty(\D^N)}^2$, which is controlled on~$[0,T]$ by its final value for bounded weak solutions of the backward Liouville equation, cf.~Definition~\ref{WeakForm}(i). Hence, by the choice~\eqref{eq:conv-PhiNT-ass} of the final condition,
\begin{equation*}
\int_{\D^N}|\Phi_N|^2f^{\otimes N}\,\le\,\|\Phi_N\|_{L^\infty(\D^N)}^2
\,\le\,\|\Phi_N^T\|_{L^\infty(\D^N)}^2\,\le\,1,
\end{equation*}
and~\eqref{eq:apriori-CN} follows.
At the final time, applying the same orthogonality argument to $\Phi_N^T-1$ yields
\[\|\Phi_N^T-1\|_{L^2(f(T)^{\otimes N})}^2=|C_{N,0}(T)-1|^2+\sum_{n=1}^N\binom{N}{n}\|C_{N,n}(T)\|_{L^2(f(T)^{\otimes n})}^2.\]
Since the left-hand side tends to $0$ by~\eqref{eq:conv-PhiNT-ass}, the convergence~\eqref{eq:final-CN} follows.
\end{proof}

For $n=2$, the a priori estimate only yields boundedness of $NC_{N,2}$ in $L^\infty(0,T;L^2(f^{\otimes2}))$, which is of course not sufficient to deduce the desired convergence~\eqref{eq:conv-CN2-weakstar}. At final time, however, \eqref{eq:final-CN} gives $NC_{N,2}(T)\to0$ in $L^2(f(T)^{\otimes2})$. We shall aim to propagate this convergence backwards in time.
In view of the a priori bounds~\eqref{eq:apriori-CN}, we introduce the fluctuation-scale dual correlations
\begin{equation}\label{eq:def-barC}
\bar C_{N,n}:=\binom Nn^{\frac12}C_{N,n},\qquad 0\le n\le N.
\end{equation}
These satisfy
\begin{equation}\label{eq:barCN-apriori}
\sup_{[0,T]}\|\bar C_{N,n}(t)\|_{L^2(f(t)^{\otimes n})}\le1,\qquad0\le n\le N,
\end{equation}
and, by~\eqref{eq:final-CN}, at the final time,
\begin{equation}\label{eq:barCN-terminal}
\bar C_{N,n}(T)\to\mathds1_{n=0}\qquad\text{in $L^2(f(T)^{\otimes n})$}.
\end{equation}
Consequently, by weak compactness and a diagonal extraction, we have for all $n\ge0$,
\begin{equation}\label{eq:weakconv}
\bar C_{N,n}\overset*\rightharpoonup\bar C_n\qquad\text{in $L^\infty(0,T;L^2(f^{\otimes n}))$},
\end{equation}
for some $\bar C_n$ satisfying
\begin{equation}\label{eq:apriori-est-barCn}
\sup_{[0,T]}\|\bar C_n\|_{L^2(f^{\otimes n})}\le1.
\end{equation}
In these terms, the desired convergence~\eqref{eq:conv-CN2-weakstar} is equivalent to showing that every weak-* limit in~\eqref{eq:weakconv} satisfies
\[\bar C_2\equiv0.\]
In order to identify the weak limits, we turn to the hierarchy of equations satisfied by the rescaled dual correlations.

\section{Limit dual hierarchy}\label{sec:dual-hier}
A fully explicit hierarchy of equations for the dual correlations $\{C_{N,n}\}_n$ can be obtained by integrating out the backward Liouville equation for $\Phi_N$. For the proof of Theorem~\ref{th:main}, however, a simpler projected hierarchy is sufficient. Indeed, $C_{N,n}$ is the orthogonal projection of the marginal~$M_{N,n}$ in~$L^2(f^{\otimes n})$ onto the subspace $\Hc_n$, cf.~\eqref{eq:def-Hn}. We may therefore derive the equation only after projection onto this subspace, or equivalently modulo a remainder in the orthogonal complement.
This eliminates a number of terms that play no role in the present argument and substantially simplifies the calculation. For later use, the full, unprojected hierarchy is recorded separately in Appendix~\ref{sec:app}.

\begin{lemma}\label{propCn}
Assume that the mean-field solution $f$ satisfies~\eqref{eq:freg-00}.
Then, for all $0\le n\le N$, we have in the distributional sense on $(0,T]\times\D^n$,
\begin{align}\nonumber
&\partial_t {C}_{N,n}  
+\sum_{i=1}^n v_i\cdot \nabla_{x_i}{C}_{N,n}
+\alpha\sum_{i=1}^n\Delta_{v_i}C_{N,n}\\
&\nonumber-\frac{N-n}{N-1}\,\sum_{j=1}^n\int_{\D} K_f(z_*,z_j)\,{C}_{N,n} (z_{[n]\setminus \{j\}},z_*)\,f(z_*)\,dz_*\\
&\nonumber+\frac{N-n}{N-1}\,\sum_{i=1}^n (K\ast f)(x_i)\cdot \nabla_{v_i}{C}_{N,n}
+ \frac{1}{N-1}\,\sum_{i\neq j}^n K(x_i-x_j)\cdot \nabla_{v_i} 
 {C}_{N,n}\\
&\nonumber-\frac{(N-n)(N-n-1)}{N-1}\int_{\D^2} K_f(z_*,z_*')\,{C}_{N,n+2}(z_{[n]},z_*,z_*')\,f(z_*)f(z_*')\,dz_*dz_*'\\
&\nonumber+\frac{N-n}{N-1}\,\sum_{i=1}^n \int_{\D} \big(K(x_i-x_{*})\,-K* f(x_i)\big)\cdot\nabla_{v_i}{C}_{N,n+1}(z_{[n]},z_*)\,f(z_*)\,dz_*\\
&\nonumber-\frac{N-n}{N-1}\,\sum_{j=1}^n \int_{\D}K_f(z_{*},z_j)\,{C}_{N,n+1}(z_{[n]},z_*) \,f(z_*)\,dz_{*}\\
&+\frac{1}{N-1}\,\sum_{i\neq j}^n K(x_i-x_j)\cdot\nabla_{v_i}{C}_{N,n-1}(z_{[n]\setminus\{j\}})\quad=\quad R_{N,n},\label{Ck} 
\end{align}
where we have set $C_{N,-1},C_{N,N+1},C_{N,N+2}\equiv0$ for notational convenience and where the remainder $R_{N,n}\in L^1(0,T;W^{-2,1}_\loc(\D^n))$ is orthogonal to $\Hc_n$ in the following weak sense: for almost all $t\in[0,T]$,
\begin{equation}\label{eq:ortho-RNn}
\int_{\D^n}h_n\,R_{N,n}(t)=0\quad \text{for all $h_n\in C^\infty_c(\D^n)$ such that $\int_{\D} h_n(z_{[n]})\,dz_j=0$ for all $1\le j\le n$}.
\end{equation}
\end{lemma} 

\begin{proof}
The assumption~\eqref{eq:freg-00} on $f$ ensures
\[K_f\in L^1(0,T;L^1(f^{\otimes2})),\qquad\int_\D |K_f(z_*,\cdot)|\,f(z_*)\,dz_*\in L^1_\loc([0,T]\times\D).\]
We first derive the BBGKY-type hierarchy of equations satisfied by the dual marginals. By definition~\eqref{eq:marginals-MNn}, combining the backward equation for~$\Phi_N$ with the mean-field equation for $f$, we obtain for all $0\le n\le N$,
\begin{align}
&\partial_tM_{N,n}
+\sum_{i=1}^n v_i\cdot\nabla_{x_i}M_{N,n}
+\alpha\sum_{i=1}^n\Delta_{v_i}M_{N,n}
\nonumber\\[-1mm]
&\nonumber\,=\,
-\frac1{N-1}\sum_{i\ne j}^nK(x_i-x_j)\cdot\nabla_{v_i}M_{N,n}\\[-1mm]
&\nonumber+\frac{N-n}{N-1}\sum_{j=1}^n\int_{\D}
K_f(z_*,z_j)\,M_{N,n+1}(z_{[n]},z_*)\, f(z_*)\,dz_*\\
&\nonumber-\frac{N-n}{N-1}\sum_{i=1}^n\int_{\D} K(x_i-x_*)\cdot\nabla_{v_i}M_{N,n+1}(z_{[n]},z_*)\, f(z_*)\,dz_*\\
&+\frac{(N-n)(N-n-1)}{N-1}\int_{\D^2} K_f(z_*,z_*')\,M_{N,n+2}(z_{[n]},z_*,z_*')\,f(z_*)f(z_*')\,dz_*dz_*',
\label{eq:BBGKY-MNn0}
\end{align}
where $M_{N,N+1},M_{N,N+2}\equiv0$. We next project this hierarchy onto the correlation space. By definition~\eqref{eq:defin-CNn}, for all $1\le n\le N$, $t\in[0,T]$, and $p_n\in L^1(\D^n)$, we have
\begin{equation}\label{eq:MCg-ortho-re}
\int_{\D^n}p_nM_{N,n}(t)\,=\,\int_{\D^n}p_nC_{N,n}(t)\quad
\text{whenever}~\int_{\D}p_n(z_{[n]})\,dz_j=0\quad\text{a.e.\@ for all $1\le j\le n$}.
\end{equation}
Let a smooth test function $h_n\in C^\infty_c([0,T]\times\D^n)$ be fixed with
\begin{equation}\label{eq:choice-hn}
\int_\D h_n(t,z_{[n]})\,dz_j = 0\quad\text{for all $1\le j\le n$}.
\end{equation}
In the distributional sense on $[0,T]$, we can decompose
\begin{multline*}
\tfrac{d}{dt}\bigg(\int_{\D^n}h_nC_{N,n}-\int_{\D^n}h_nM_{N,n}\bigg)
=\bigg(\int_{\D^n}h_n\partial_tC_{N,n}-\int_{\D^n}h_n\partial_tM_{N,n}\bigg)\\
+\bigg(\int_{\D^n}C_{N,n}\partial_th_n-\int_{\D^n}M_{N,n}\partial_th_n\bigg).
\end{multline*}
By~\eqref{eq:MCg-ortho-re} and by the choice~\eqref{eq:choice-hn} of $h_n$, both the left-hand side and the second right-hand side term vanish, hence
\begin{equation*}
\int_{\D^n}h_n\partial_tC_{N,n}\,=\,\int_{\D^n}h_n\partial_tM_{N,n}.
\end{equation*}
Inserting equation~\eqref{eq:BBGKY-MNn0}, we get
\begin{equation}\label{eq:proof-eqnCNn}
\int_{\D^n} h_n\partial_t{C}_{N,n}
\,=\,T_1+T_2+\frac{N-n}{N-1}(T_3+T_4)+\frac{(N-n)(N-n-1)}{N-1}T_5,
\end{equation}
where we have set for abbreviation
\begin{eqnarray*}
T_1 &\!:=\!&\int_{\D^n}M_{N,n}\bigg(\sum_{i=1}^n v_i \cdot \nabla_{x_i}+\frac{1}{N-1} \sum_{i\ne j}^n K(x_i-x_j)\cdot\nabla_{v_i}\bigg)h_n,\\[-1mm]
T_2&\!:=\!&-\alpha\sum_{i=1}^n\int_{\D^n} M_{N,n}\,\Delta_{v_i}h_n,\\[-1mm]
T_3&\!:=\!&\sum_{j=1}^n\int_{\D^{n+1}}K_f(z_{n+1},z_j)\,h_n(z_{[n]})M_{N,n+1}(z_{[n+1]})f(z_{n+1})\,dz_{[n+1]},\\[-1mm]
T_4&\!:=\!&\sum_{i=1}^n\int_{\D^{n+1}}K(x_i-x_{n+1})\cdot\nabla_{v_i} h_n(z_{[n]}) M_{N,n+1}(z_{[n+1]})\,f(z_{n+1})\,dz_{[n+1]},\\[-1mm]
T_5&\!:=\!&\int_{\D^{n+2}}K_f(z_{n+1},z_{n+2})h_n(z_{[n]})M_{N,n+2}(z_{[n+2]})\,f(z_{n+1})f(z_{n+2})\,dz_{[n+2]}.
\end{eqnarray*}
In order to derive the desired equation for $C_{N,n}$, it remains to re-express these five quantities $\{T_j\}_{1\le j\le 5}$ in terms of correlations instead of marginals.
Inserting the cluster expansion~\eqref{eq:clusterexpPhin}, using both~\eqref{eq:choice-hn} and the cancellation properties of $K_f$, cf.~\eqref{eq:cancel-Vf}, we find
\begin{eqnarray*}
T_1 &\!=\!&\int_{\D^n}C_{N,n}\bigg(\sum_{i=1}^n v_i \cdot \nabla_{x_i}+\frac{1}{N-1} \sum_{i\ne j}^n K(x_i-x_j)\cdot\nabla_{v_i}\bigg)h_n\\[-1mm]
&&+\frac{1}{N-1} \sum_{i\ne j}^n\int_{\D^n}C_{N,n-1}(z_{[n]\setminus\{j\}})K(x_i-x_j)\cdot\nabla_{v_i}h_n,\\
T_2&\!=\!&-\alpha\sum_{i=1}^n\int_{\D^n} C_{N,n}\,\Delta_{v_i}h_n,\\
T_3&\!=\!&\sum_{j=1}^n\int_{\D^{n+1}}K_f(z_{n+1},z_j)\,h_n(z_{[n]})\,C_{N,n+1}(z_{[n+1]})f(z_{n+1})\,dz_{[n+1]}\\[-2mm]
&&+\sum_{j=1}^n\int_{\D^{n+1}}K_f(z_{n+1},z_j)\,h_n(z_{[n]})\,C_{N,n}(z_{[n+1]\setminus\{j\}})f(z_{n+1})\,dz_{[n+1]},\\
T_4&\!=\!&\sum_{i=1}^n\int_{\D^{n+1}}K(x_i-x_{n+1})\cdot\nabla_{v_i} h_n(z_{[n]})\,C_{N,n+1}(z_{[n+1]})\,f(z_{n+1})\,dz_{[n+1]}\\[-1mm]
&&+\sum_{i=1}^n\int_{\D^{n}}K\ast f(x_i)\cdot\nabla_{v_i} h_n(z_{[n]}) \,C_{N,n}(z_{[n]})\,dz_{[n]},\\[1mm]
T_5&\!=\!&\int_{\D^{n+2}}K_f(z_{n+1},z_{n+2})\,h_n(z_{[n]})\,C_{N,n+2}(z_{[n+2]})\,f(z_{n+1})f(z_{n+2})\,dz_{[n+2]}.
\end{eqnarray*}
Combining these identities in~\eqref{eq:proof-eqnCNn}, and defining $R_{N,n}$ as the left-hand side of~\eqref{Ck}, we obtain in the distributional sense on $[0,T]$,
\[\int_{\D^n}h_nR_{N,n}\,=\,0,\]
that is, \eqref{eq:ortho-RNn}. Finally, by~\eqref{eq:BBGKY-MNn0}, the boundedness of the marginals and the assumed integrability of $K$ and $\nabla_v f$ give $\partial_t M_{N,n}\in L^1(0,T;W^{-2,1}_{\loc}(\D^n))$, hence also $\partial_t C_{N,n}\in L^1(0,T;W^{-2,1}_{\loc}(\D^n))$. As the other left-hand side terms in~\eqref{Ck} have the same negative Sobolev regularity, we conclude $R_{N,n}\in L^1(0,T;W^{-2,1}_{\loc}(\D^n))$.
\end{proof}

We now pass to the limit in the above projected hierarchy. A priori, this should only determine the limit equation modulo a remainder in $\Hc_n^\bot$. Yet, the limit hierarchy actually propagates the structure of~$\Hc_n$, so that the remainder vanishes in the limit. Alternatively, this could be deduced from the explicit expression of the remainder in Appendix~\ref{sec:app}, but the present argument is computationally lighter.

\begin{lemma}\label{propbarCn}
Assume that $K$ satisfies~\eqref{hypothesis-Kernel} and that $f$ satisfies~\eqref{hypothesis-f0}--\eqref{hypothesis-f}, and let $\{\bar C_n\}_n$ be any family of limit rescaled correlations extracted in~\eqref{eq:weakconv}. Then $\bar C_n(t)\in\Hc_n(t)$ for almost all $t\in[0,T]$ and all $n\ge0$, and the following hierarchy is satisfied in the weighted weak sense
\begin{multline}\label{eq:lim-hier1}
\partial_t\bar C_{n}
+\sum_{i=1}^nv_i\cdot\nabla_{x_i}\bar C_{n}
+\sum_{i=1}^n (K\ast f)(x_i)\cdot\nabla_{v_i}\bar C_{n}
+\alpha\sum_{i=1}^n\Delta_{v_i}\bar C_n\\
\,= \,\sum_{j=1}^n\int_\D K_f(z_*,z_j)\,\bar C_{n}(z_{[n]\setminus\{j\}},z_*)\,f(z_*)\,dz_*\\
+\sqrt{n+1}\,\sqrt{n+2}\int_{\D^2} K_f(z_*,z_*')\,\bar C_{n+2}(z_{[n]},z_*,z_*')\,f(z_*)f(z_*')\,dz_*dz_*',
\end{multline}
with final condition $\bar C_{n}|_{t=T}=\mathds1_{n=0}$. More precisely, denoting by $S_n$ the right-hand side of~\eqref{eq:lim-hier1}, we have for all $g_n\in C^\infty_c((0,T]\times\D^n)$,
\begin{multline}\label{eq:projected-weighted-weak}
g_0(T)\mathds1_{n=0}
=\int_0^T\!\int_{\D^n}\bar C_nf^{\otimes n}
\Big(\partial_tg_n+\sum_{i=1}^nv_i\cdot\nabla_{x_i}g_n+\sum_{i=1}^n(K\ast f)(x_i)\cdot\nabla_{v_i}g_n-\alpha\sum_{i=1}^n\Delta_{v_i}g_n\Big)\\
-2\alpha\sum_{i=1}^n\int_0^T\!\int_{\D^n}\bar C_n\,\nabla_{v_i}g_n\cdot\nabla_{v_i}f^{\otimes n}
+\int_0^T\!\int_{\D^n}S_ng_nf^{\otimes n}.
\end{multline}
\end{lemma}

\begin{proof}
The centering~\eqref{eq:orthogonality} passes to the weak limit, hence $\bar C_n(t)\in\Hc_n(t)$.
We turn to the proof of~\eqref{eq:projected-weighted-weak} and we start by establishing it for centered test functions of the form \begin{equation}\label{eq:projected-tests}
g_n(t):=(\Id-\pi_{f(t)})^{\otimes n}\varphi_n(t),\qquad \varphi_n\in C_c^\infty((0,T]\times\D^n).
\end{equation}
Such test functions satisfy $\int_\D g_n(z_{[n]}) f(z_j)\,dz_j=0$ for all $j$, and thus $\int_{\D^n}M_{N,n}g_nf^{\otimes n}=\int_{\D^n}C_{N,n}g_nf^{\otimes n}$.
Due to the time-dependent projections, $g_n$ need not be compactly supported or smooth in time, but the mean-field equation for $f$ ensures that it is absolutely continuous in time. Starting directly from the finite-$N$ marginal hierarchy~\eqref{eq:BBGKY-MNn0}, using the cluster expansion and the same algebraic cancellations as in the proof of Lemma~\ref{propCn}, together with the mean-field equation for $f$, we obtain, after a standard approximation argument,
\begin{multline}\label{eq:prelim-CNn}
\int_{\D^n}C_{N,n}(T)g_n(T)f(T)^{\otimes n}
=\int_0^T\!\int_{\D^n}C_{N,n}f^{\otimes n}\Big(\partial_tg_n
+\sum_{i=1}^n\big(v_i\cdot\nabla_{x_i}+(K\ast f)(x_i)\cdot\nabla_{v_i}-\alpha\Delta_{v_i}\big)g_n\Big)\\
-2\alpha\sum_{i=1}^n\int_0^T\!\int_{\D^n}C_{N,n}\nabla_{v_i}g_n\cdot\nabla_{v_i}f^{\otimes n}
+\int_0^T\!\int_{\D^n}S_{N,n}g_nf^{\otimes n}
+\int_0^T\!\int_{\D^n}E_{N,n}(g_n)f^{\otimes n},
\end{multline}
where we have set for abbreviation
\begin{eqnarray*}
S_{N,n}&:=&\frac{N-n}{N-1}\sum_{j=1}^n\int_\D K_f(z_*,z_j) C_{N,n}(z_{[n]\setminus\{j\}},z_*)f(z_*)\,dz_*\\
&+&\frac{(N-n)(N-n-1)}{N-1}\int_{\D^2}K_f(z_*,z_*') C_{N,n+2}(z_{[n]},z_*,z_*')f(z_*)f(z_*')\,dz_*dz_*',\\
E_{N,n}(g_n)
&=&\frac{C_{N,n}}{N-1}\sum_{i\ne j}^n
 (K(x_i-x_j)-(K\ast f)(x_i))\cdot\nabla_{f,v_i}g_n\\
&+&\frac{N-n}{N-1}\sum_{i=1}^n\nabla_{f,v_i}g_n\cdot \int_\D(K(x_i-x_*)-(K\ast f)(x_i))C_{N,n+1}(z_{[n]},z_*)f(z_*)\,dz_*\\
&+&\frac{N-n}{N-1}\,g_n\sum_{j=1}^n\int_\D K_f(z_*,z_j) C_{N,n+1}(z_{[n]},z_*)f(z_*)\,dz_*\\
&+&\frac1{N-1}\sum_{i\ne j}^n C_{N,n-1}(z_{[n]\setminus\{j\}})K(x_i-x_j)\cdot\nabla_{f,v_i}g_n,
\end{eqnarray*}
with the short-hand notation $\nabla_{f,v_i}g:=\nabla_{v_i}g+g(\nabla_{v}\log f)(z_i)$ (always interpreted after multiplication by $f^{\otimes n}$). Note that all terms are well-defined by the assumptions on $K,f$, only using $C_{N,n}\in L^\infty(0,T;L^2(f^{\otimes n}))$.

Now multiply~\eqref{eq:prelim-CNn} by $\binom Nn{}^{1/2}$. For fixed $n$, the terms in $E_{N,n}(g_n)$ involving $\bar C_{N,n\pm1}$ have coefficients $O_n(N^{-1/2})$, while the term involving $\bar C_{N,n}$ has coefficient $O_n(N^{-1})$. By the assumptions on $K,f$, the bound~\eqref{eq:barCN-apriori} therefore gives, for any test function of the form~\eqref{eq:projected-tests},
\[\binom Nn^{\frac12}\int_0^T\!\int_{\D^n}E_{N,n}(g_n)f^{\otimes n}\to0.\]
After rescaling, the term involving $\bar C_{N,n+2}$ in $S_{N,n}$ has coefficient
\[\frac{(N-n)(N-n-1)}{N-1}\binom Nn^{1/2}\binom N{n+2}^{-1/2}\to\sqrt{n+1}\sqrt{n+2}.\]
As the assumptions on $K,f$ ensure $K_f\in L^1(0,T;L^2(f^{\otimes2}))$, all the other terms pass to the limit along~\eqref{eq:weakconv}. By the convergence of final data~\eqref{eq:barCN-terminal}, this concludes the proof of~\eqref{eq:projected-weighted-weak} for test functions of the form~\eqref{eq:projected-tests}.

Finally, the identity extends to arbitrary $g_n\in C_c^\infty((0,T]\times\D^n)$. Indeed, since $\bar C_n(t)\in\Hc_n(t)$, by the cancellation properties~\eqref{eq:cancel-Vf} for $K_f$, we also have $S_n(t)\in\Hc_n(t)$. Hence, each right-hand side term in~\eqref{eq:projected-weighted-weak} vanishes when $g_n$ is independent of at least one particle variable. For a smooth test function~$\varphi_n$, the difference $\varphi_n-(\Id-\pi_{f})^{\otimes n}\varphi_n$ is a sum of functions independent of at least one particle variable. As~\eqref{eq:projected-weighted-weak} already holds for $g_n=(\Id-\pi_{f})^{\otimes n}\varphi_n$, it therefore holds for $g_n=\varphi_n$ as well.
\end{proof}

We have thus obtained a closed limit hierarchy for limit correlations $\{\bar C_n\}_n$. By~\eqref{eq:cancel-Vf}, the constant family $\bar C_n\equiv\mathds1_{n=0}$ is a trivial solution of this hierarchy with the same final condition. In order to identify the weak limit, it remains to establish uniqueness.

The coefficients in the limit hierarchy~\eqref{eq:lim-hier1} grow linearly with the correlation order $n$, so a uniform energy estimate over all levels does not close directly. In the present $L^2$ setting, however, uniqueness follows simply from a generating-function argument of Cauchy-Kowalevskaya type, as classically used for infinite BBGKY hierarchies; see e.g.~\cite[Section~II.2]{Golse16}.

\begin{lemma}[Uniqueness for the limit hierarchy]\label{lem:unique-hier}
Assume that $K$ satisfies~\eqref{hypothesis-Kernel} and that the mean-field solution~$f$ satisfies~\eqref{hypothesis-f0}--\eqref{hypothesis-f}.
Let $\{C_n\}_n$ satisfy the limit hierarchy~\eqref{eq:lim-hier1} in the weighted weak sense~\eqref{eq:projected-weighted-weak}, with vanishing final data
\[C_n|_{t=T}=0\qquad\text{for all $n\ge0$},\]
with $C_n(t)\in\Hc_n(t)$ for almost all $t$, and with the a priori bound
\begin{equation}\label{eq:Cn'-bnd}
\sup_{n\ge0}\,\underset{[0,T]}{\operatorname{ess\,sup}}\,\|C_n\|_{L^2(f^{\otimes n})}\,<\,\infty.
\end{equation}
Then $C_n\equiv0$ for all $n\ge0$.
\end{lemma}

\begin{proof}
Let $\{C_n\}_n$ satisfy the limit hierarchy~\eqref{eq:lim-hier1} in the weighted weak sense with vanishing final data, and let $S_n$ denote its right-hand side.
We first justify the following weighted energy estimate, in the distributional sense on~$(0,T]$,
\begin{equation}\label{eq:weighted-energy-unique}
\frac12\frac{d}{dt}\int_{\D^n}|C_n|^2f^{\otimes n}\ge\int_{\D^n}C_nS_nf^{\otimes n}.
\end{equation}

For $\alpha=0$, the assumptions on $K,f$ ensure $K\ast f\in L^\infty([0,T]\times\Omega)\cap L^1(0,T;H^1_\loc(\Omega))$. Hence the divergence-free vector field $B(t,x,v):=(v,(K\ast f)(x,t))$ belongs to $L^1(0,T;H^1_\loc(\D))$ and has at most linear growth. To apply the DiPerna-Lions theory without assuming positivity of $f$, consider the equations for
\[u_n:=C_nf^{\otimes n}\qquad\text{and}\qquad w_n:=f^{\otimes n}.\]
Denoting by $B_n$ the tensorized vector field on $\D^n$, the weighted weak formulation for~\eqref{eq:lim-hier1} and the equation for~$f$ give
\[\partial_tu_n+\Div(B_nu_n)=S_nw_n,\qquad\partial_tw_n+\Div(B_nw_n)=0.\]
Since $u_n\in L^\infty(0,T;L^2(\Dd^n))$, $w_n\in L^\infty([0,T]\times\Dd^n)$, and $S_nw_n\in L^1(0,T;L^2(\D^n))$, the regularity of $B$ allows to apply the DiPerna-Lions commutator argument~\cite{DiLi} jointly to these equations. More precisely, let $\theta_M$ be a smooth convex truncation of $r\mapsto\frac12 r^2$, with linear growth, and consider $\Phi_M(u,w):=w\theta_M(u/w)$. The possible vanishing of $w$ causes no issue since $\Phi_M$ has bounded first derivatives and can be extended across $w=0$ by approximation. For $\Phi_M(u_n,w_n)=\theta_M(C_n)f^{\otimes n}$, joint renormalization then yields
\[\partial_t(\theta_M(C_n)f^{\otimes n})+\Div(B_n\theta_M(C_n)f^{\otimes n})=\theta_M'(C_n)S_nf^{\otimes n}.\]
Integrating in space and letting $M\uparrow\infty$ then proves~\eqref{eq:weighted-energy-unique} (with equality).

For $\alpha>0$, we similarly use parabolic renormalization. Since $K\ast f$ is bounded, standard positivity and velocity-gradient estimates for kinetic Fokker-Planck equations ensure that $w_n=f^{\otimes n}$ is locally bounded away from zero and that $\nabla_v\log w_n$ is locally bounded on $(0,T]\times\D^n$. For $u_n=C_nf^{\otimes n}$, the weighted weak equation takes the form
\[\partial_tu_n+\Div(B_nu_n)=-\alpha\Delta_vu_n+2\alpha\Div_v(u_n\nabla_v\log w_n)+S_nw_n,\]
and a standard local energy estimate for $L^2$ solutions yields $\nabla_vu_n,\nabla_vC_n\in L^2_\loc((0,T]\times\D^n)$. This is sufficient to justify the parabolic chain rule by approximation:
\begin{multline*}
\partial_t(\theta_M(C_n)f^{\otimes n})+\Div(B_n\theta_M(C_n)f^{\otimes n})+\alpha\Delta_v(\theta_M(C_n)f^{\otimes n})-2\alpha\Div_v(\theta_M(C_n)\nabla_vf^{\otimes n})\\
=\theta_M'(C_n)S_nf^{\otimes n}
+\alpha \theta_M''(C_n)|\nabla_vC_n|^2f^{\otimes n}.
\end{multline*}
The last term is nonnegative and the claim~\eqref{eq:weighted-energy-unique} follows by integrating in space and letting $M\uparrow\infty$.

We may now conclude the uniqueness proof. Recalling that $S_n$ is defined as the right-hand side in~\eqref{eq:lim-hier1}, we obtain from~\eqref{eq:weighted-energy-unique}, by Cauchy-Schwarz,
\[\frac{d}{dt}\|C_n\|_{L^2(f^{\otimes n})}\ge-\Lambda\Big(n\|C_n\|_{L^2(f^{\otimes n})}+(n+2)\|C_{n+2}\|_{L^2(f^{\otimes n+2})}\Big),\]
with the short-hand notation $\Lambda:=\|K_f\|_{L^2(f^{\otimes2})}$.
As in the standard Cauchy-Kowalevskaya argument, introduce the generating function
\[Z(t,r):=\sum_{n\ge0}r^n\|C_n(t)\|_{L^2(f(t)^{\otimes n})},\]
which is finite for almost all $t\in[0,T]$ and $0\le r<1$ by~\eqref{eq:Cn'-bnd}, and vanishes at $t=T$.
In these terms, we find
\begin{eqnarray*}
\partial_tZ(t,r)&\ge&-\Lambda(t)\sum_{n\ge0}r^n\Big(n\|C_n(t)\|_{L^2(f(t)^{\otimes n})}+(n+2)\|C_{n+2}(t)\|_{L^2(f(t)^{\otimes n+2})}\Big)\\
&\ge&-\Lambda(t)(r+r^{-1})\partial_rZ(t,r)
\,\ge\,-2\Lambda(t)r^{-1}\partial_rZ(t,r).
\end{eqnarray*}
Solving this differential inequality along characteristics and using $Z(T,\cdot)=0$,
we deduce for all $r\in[0,1)$ and $T_0\in[0,T]$, provided that $r^2+4\int_{T_0}^T\Lambda<1$,
\[0\,\le\, Z(T_0,r)\,\le\, Z\Big(T,\big(r^2+4{\textstyle\int_{T_0}^T\Lambda}\big)^\frac12\Big)\,=\,0.\]
Since $\Lambda\in L^1(0,T)$, we may choose $T_0<T$ such that $\int_{T_0}^T\Lambda<\frac18$. It follows that $Z(t,r)=0$ for all $t\in[T_0, T]$ and $r\in[0,\frac12]$, hence $C_n(t)=0$ for all $n\ge0$ and $t\in[T_0, T]$.
Iterating the argument over finitely many successive time intervals, we conclude $C_n\equiv0$ on $[0,T]$ for all $n\ge0$.
\end{proof}

Combining the different observations above, we are now in position to conclude the proof of Theorem~\ref{th:main}.

\begin{proof}[Proof of Theorem~\ref{th:main}]
By Lemma~\ref{propbarCn}, every family of limit rescaled correlations $\{\bar C_n\}_n$ extracted in~\eqref{eq:weakconv} satisfies the limit hierarchy~\eqref{eq:lim-hier1} in the weighted weak sense, as well as the a priori bound~\eqref{eq:apriori-est-barCn}, with the final condition $\bar C_n|_{t=T}=\mathds1_{n=0}$.
By~\eqref{eq:cancel-Vf}, the constant family $C_n\equiv\mathds1_{n=0}$ is a trivial solution of this hierarchy with the same a priori bound and final condition. Lemma~\ref{lem:unique-hier} therefore implies that every extracted limit $\{\bar C_n\}_n$ coincides with this solution. In particular,
\[\bar C_2\equiv0,\]
and hence, independently of subsequences,
\[NC_{N,2}\overset*\rightharpoonup0\qquad\text{in $L^\infty(0,T;L^2(f^{\otimes2}))$.}\]
Since $K_f\in L^1(0,T;L^2(f^{\otimes2}))$, equation~\eqref{eq:rewr-red} then yields for all $t\in[0,T]$,
\begin{equation}
N\!\int_t^T\!\Big(\int_{\D^N}K_f(z_1,z_2)\,\Phi_N\,f^{\otimes N}\Big)\to0,
\end{equation}
and the conclusion follows from Proposition~\ref{lem:duality}.
\end{proof}

\section{Extensions}\label{sec:extensions}
We briefly explain how the duality method adapts to the extensions stated in Section~\ref{sec:gen}. For general compressible interactions, the main new feature is the appearance of divergence terms in the dual hierarchy. In the first-order setting, one can in addition exploit spatial symmetrization to partially compensate for the singularity of the interaction kernel.

Note that the use of measure-valued weak duality solutions does not bring any complications: indeed, the duality formula again reduces the mean-field limit to the study of a bounded weak solution of the backward Liouville equation, independently of possible concentration on the primal side.

\subsection{First-order dynamics}
We consider the first-order dynamics~\eqref{eq:SDE-1st} and only indicate the main modifications needed to prove Theorem~\ref{th:main-1st}. The construction of dual correlations and the $L^2$ orthogonality estimates are unchanged. The analogue of Lemma~\ref{propCn} takes the following form.

\begin{lemma}\label{propCn0}
Under the assumptions of Theorem~\ref{th:main-1st}, for all $0\le n\le N$, we have in the distributional sense on $[0,T]\times\Omega^n$,
\begin{align}\label{Ck0} 
&\partial_t {C}_{N,n}  
+\alpha\sum_{i=1}^n\Delta_{x_i}C_{N,n}\\
&\nonumber-\frac{N-n}{N-1}\,\sum_{j=1}^n\int_{\Omega} K_f(x_*,x_j)\,C_{N,n} (x_{[n]\setminus \{j\}},x_*)\,f(x_*)\,dx_*\\
&\nonumber+\frac{N-n}{N-1}\,\sum_{i=1}^n (K\ast f)(x_i)\cdot \nabla_{x_i}{C}_{N,n}
+ \frac{1}{N-1}\,\sum_{i\neq j}^n K(x_i-x_j)\cdot \nabla_{x_i} {C}_{N,n}\\
&\nonumber-\frac{(N-n)(N-n-1)}{N-1}\int_{\Omega^2} K_f(x_*,x_*')\,{C}_{N,n+2}(x_{[n]},x_*,x_*')\,f(x_*)f(x_*')\,dx_*dx_*'\\
&\nonumber+\frac{N-n}{N-1}\,\sum_{i=1}^n\int_{\Omega} \big(K(x_i-x_{*})\,-K* f(x_i)\big)\cdot\nabla_{x_i}{C}_{N,n+1}(x_{[n]},x_*)\,f(x_*)\,dx_*\\
&\nonumber-\frac{N-n}{N-1}\,\sum_{j=1}^n \int_{\Omega}K_f(x_{*},x_j)\,{C}_{N,n+1}(x_{[n]},x_*) \,f(x_*)\,dx_{*}\\
&\nonumber+\frac{1}{N-1}\,\sum_{i\neq j}^n K(x_i-x_j)\cdot\nabla_{x_i}{C}_{N,n-1}(x_{[n]\setminus\{j\}})\quad=\quad R_{N,n},
\end{align}
in terms of
\begin{equation}\label{eq:redef-Kf}
K_f(x,x'):=\frac1{f(x)}\,\Div_{x}\Big(\big(K(x-x')-K\ast f(x)\big)f(x)\Big),
\end{equation}
for some $R_{N,n}\in L^1(0,T;W^{-2,1}_\loc(\Omega^n))$ that is orthogonal to $\Hc_n$ in the weak sense as in~\eqref{eq:ortho-RNn}.
\end{lemma}

The proof is identical to that of Lemma~\ref{propCn}, except that integration by parts now produces the additional divergence terms contained in $K_f$, cf.~\eqref{eq:redef-Kf}. Note that this quantity satisfies the same cancellation properties~\eqref{eq:cancel-Vf} as before.
The main remaining point is to check that the assumptions of Theorem~\ref{th:main-1st} are sufficient to pass to the limit in the corresponding weak formulation. This requires some particular care as the coefficient $K_f$ itself need not belong to $L^1(0,T;L^2(f^{\otimes2}))$ under the assumptions of Theorem~\ref{th:main-1st}: only suitable symmetrizations or combinations of the terms in the hierarchy enjoy the required square-integrability. We separately consider the different interaction terms in~\eqref{Ck0}. By symmetry, we may restrict to symmetric test functions throughout.

\medskip\noindent
$\bullet$ {\it Same-level interaction terms.}
Consider
\begin{eqnarray*}
T_{N,n,j}^1&:=&\int_\Omega K_f(x_*,x_j)\,C_{N,n}(x_{[n]\setminus\{j\}},x_*)\,f(x_*)\,dx_*,\\
T_{N,n,ij}^2&:=&K(x_i-x_j)\cdot\nabla_{x_i}C_{N,n}.
\end{eqnarray*}
We start with $T_{N,n,j}^1$. Given $h_n\in C^\infty_c([0,T]\times\Omega^n)$, testing with $h_nf^{\otimes n}$, and exchanging $x_j$ and $x_{*}$,
we get
\[\int_{\Omega^n}h_nT_{N,n,j}^1f^{\otimes n}=\int_{\Omega^n} A^1_{n,j}C_{N,n}f^{\otimes n},\]
in terms of
\[A^1_{n,j}(x_{[n]}):=\int_{\Omega}K_f(x_j,x_*)\,h_n(x_{[n]\setminus\{j\}},x_*)\,f(x_*)\,dx_*.\]
Since $K,\Div(K)\in L^1_\loc(\Omega)$ are bounded at infinity and since $f$ is bounded and satisfies the Fisher-information bound $\nabla\log f\in L^1(0,T;L^2(f^{\otimes2}))$, we deduce $A_{n,j}^1\in L^1(0,T;L^2(f^{\otimes n}))$, which is sufficient to pass to the limit in this term.

\medskip\noindent
We turn to $T_{N,n,ij}^2$. Testing with $h_nf^{\otimes n}$, integrating by parts, symmetrizing in $(x_i,x_j)$, and decomposing $K=K_0+K_1$ into its even and odd parts, we get
\begin{equation*}
\int_{\Omega^n}h_nT^2_{N,n,ij}f^{\otimes n}
=-\int_{\Omega^n}A^2_{n,ij}C_{N,n}f^{\otimes n},
\end{equation*}
in terms of
\begin{multline*}
A^2_{n,ij}:=K_0(x_i-x_j)\cdot\Big(\nabla_{x_i}h_n+h_n\nabla\log f(x_i)\Big)
+\frac12K_1(x_i-x_j)\cdot(\nabla_{x_i}h_n-\nabla_{x_j}h_n)\\
+\frac12K_1(x_i-x_j)\cdot(\nabla\log f(x_i)-\nabla\log f(x_j))\,h_n
+\Div(K_1)(x_i-x_j)\,h_n.
\end{multline*}
Noting that the symmetry of $h_n$ entails
\[|\nabla_{x_i}h_n-\nabla_{x_j}h_n|\le|x_i-x_j|\|\nabla^2h_n\|_{L^\infty},\]
and recalling the assumptions $K_0\in L^2_\loc(\Omega;\R^d)$, $|x|K_1\in L^2_\loc(\Omega;\R^d)$, and $\Div(K_1)\in L^2_\loc(\Omega)$, together with the Fisher-information bound for $f$ and the assumption on $H_f$, cf.~\eqref{eq:def-Hf}, we deduce $A^2_{n,ij}\in L^1(0,T;L^2(f^{\otimes n}))$. This allows to pass to the limit in this term, which vanishes after rescaling.

\medskip\noindent
$\bullet$ {\it Interaction term involving $C_{N,n+2}$.} Consider
\[T_{N,n}^3:=\int_{\Omega^2} K_f(x_*,x_*')\,{C}_{N,n+2}(x_{[n]},x_*,x_*')\,f(x_*)f(x_*')\,dx_*dx_*'.\]
Symmetrizing in $(x_*,x_*')$, we can replace $K_f$ by
\[K_f^\Sym(x,x'):=\frac12\big(K_f(x,x')+K_f(x',x)\big).\]
Using~\eqref{eq:Kfsym}, the assumptions of Theorem~\ref{th:main-1st} are precisely designed to ensure $K_f^\Sym\in L^1(0,T;L^2(f^{\otimes2}))$, which then allows to pass to the limit in this term.

\medskip\noindent
$\bullet$ \textit{Interaction terms involving $C_{N,n+1}$.}
The two terms involving $C_{N,n+1}$ in~\eqref{Ck0} have to be estimated together,
\begin{multline*}
T_{N,n,j}^4:=\int_{\Omega} \big(K(x_j-x_{*})\,-K* f(x_j)\big)\cdot\nabla_{x_j}{C}_{N,n+1}(x_{[n]},x_*)\,f(x_*)\,dx_*\\
-\int_{\Omega}K_f(x_{*},x_j)\,{C}_{N,n+1}(x_{[n]},x_*) \,f(x_*)\,dx_{*}.
\end{multline*}
Testing with $h_nf^{\otimes n}$ and integrating by parts in the first term, we find
\begin{multline*}
\int_{\Omega^n}h_nT_{N,n,j}^4f^{\otimes n}
=-\int_{\Omega^{n+1}}\Big(\big(K_f(x_j,x_{n+1})+K_f(x_{n+1},x_j)\big)h_n(x_{[n]})\\
+\big(K(x_j-x_{n+1})\,-K* f(x_j)\big)\cdot\nabla_{x_j}h_n(x_{[n]})\Big){C}_{N,n+1}f^{\otimes n+1}.
\end{multline*}
Thanks to this recombination, we recognize the symmetrized kernel $K_f^\Sym$ in the first term. Symmetrizing the second term with respect to $(x_j,x_{n+1})$, and decomposing $K=K_0+K_1$ into its even and odd parts, we obtain
\[\int_{\Omega^n}h_nT_{N,n,j}^4f^{\otimes n}=-\int_{\Omega^{n+1}}A_{n,j}^4C_{N,n+1}f^{\otimes n+1},\]
in terms of
\begin{multline*}
A_{n,j}^4(x_{[n+1]}):=2K_f^\Sym(x_j,x_{n+1})\,h_n(x_{[n]})
+K_0(x_j-x_{n+1})\cdot\nabla_{x_j}h_n(x_{[n]})\\
+\frac12K_1(x_j-x_{n+1})\cdot\big(\nabla_{x_j}h_n(x_{[n]})-\nabla_{x_{n+1}}h_n(x_{[n+1]\setminus\{j\}})\big)
-K* f(x_j)\cdot\nabla_{x_j}h_n(x_{[n]}).
\end{multline*}
As the symmetry of $h_n$ entails
\[\big|\nabla_{x_j}h_n(x_{[n]})-\nabla_{x_{n+1}}h_n(x_{[n+1]\setminus\{j\}})\big|\le|x_j-x_{n+1}|\|\nabla^2h_n\|_{L^\infty},\]
the assumptions of Theorem~\ref{th:main-1st} ensure $A_{n,j}^4\in L^1(0,T;L^2(f^{\otimes n+1}))$. This allows to pass to the limit in this term, which vanishes after rescaling.

\medskip\noindent
$\bullet$ \textit{Interaction term involving $C_{N,n-1}$.}
Finally, consider
\[T_{N,n,ij}^5:=K(x_i-x_j)\cdot\nabla_{x_i}C_{N,n-1}(x_{[n]\setminus\{j\}}).\]
Testing with $h_nf^{\otimes n}$, integrating by parts in $x_i$, and first integrating out $x_j$ since $C_{N,n-1}$ is independent of that variable, we get
\[\int_{\Omega^n}h_nT_{N,n,ij}^5f^{\otimes n}=-\int_{\Omega^{n-1}}A_{n,ij}^5C_{N,n-1}f^{\otimes(n-1)},\]
where
\begin{multline*}
A_{n,ij}^5(x_{[n]\setminus\{j\}}):=\int_\Omega\Big(K(x_i-x_j)\cdot\nabla_{x_i}h_n(x_{[n]})+\Div(K)(x_i-x_j)h_n(x_{[n]})\Big)f(x_j)\,dx_j\\
+\nabla\log f(x_i)\cdot\int_\Omega K(x_i-x_j) h_n(x_{[n]})\,f(x_j)\,dx_j.
\end{multline*}
Since $K,\Div(K)\in L^1_\loc(\Omega)$ are bounded at infinity and since $f$ is bounded and satisfies the Fisher-information bound $\nabla\log f\in L^1(0,T;L^2(f))$, we deduce $A_{n,ij}^5\in L^1(0,T;L^2(f^{\otimes n-1}))$. This allows to pass to the limit in this term, which vanishes after rescaling.

\medskip
We have thus obtained the estimates needed to pass to the limit in the projected hierarchy~\eqref{Ck0} and
the argument of Lemma~\ref{propbarCn} applies in the same way.
Finally, uniqueness of the limit hierarchy follows as in Lemma~\ref{lem:unique-hier} and allows to identify the limit dual correlations. When $\alpha=0$, the required regularity $K\ast f\in L^1(0,T;H^1_\loc(\Omega))$ follows from the boundedness of $f$, the assumption $\nabla\sqrt f\in L^1(0,T;L^2(\Omega))$, and $K\in L^1_\loc(\Omega)$ bounded at infinity.

With uniqueness at hand for the limit hierarchy, the remainder of the proof of the mean-field limit is identical to that of Theorem~\ref{th:main}.

\subsection{General phase-space interactions}
We finally comment on the general phase-space interactions considered in Section~\ref{sec:gen}. The proof is the same as for the second-order dynamics, except that the interaction force may now be compressible in the velocity variable. Accordingly, the singular two-particle coefficient entering the dual hierarchy is replaced by
\[K_f(z,z'):=\frac1{f(z)}\Div_v\Big(\big(K(z,z')-K[f](z)\big)f(z)\Big),\]
that is,
\[K_f(z,z')=\Div_v(K)(z,z')-\Div_v(K[f])(z)+\big(K(z,z')-K[f](z)\big)\cdot\nabla_v\log f(z).\]
The assumptions in Theorem~\ref{th:geninter} ensure $K_f\in L^1(0,T;L^2(f^{\otimes2}))$ and allow to control all the other terms appearing in the hierarchy. With these replacements, the compactness argument and the uniqueness proof proceed exactly as in the second-order case. We omit the details.

\bigskip
\subsection*{Acknowledgments:}
D. Bresch was partially supported by the BOURGEONS project, grant  ANR-23-CE40-0014-01 of the French National Research Agency (ANR). This work also benefited from the support of the ANR under France 2030, grant ANR-23-EXMA-004 (Complexflows project). 
M. Duerinckx acknowledges financial support from the European Union (ERC, PASTIS, Grant Agreement\linebreak n$^\circ$101075879).\footnote{{Views and opinions expressed are however those of the authors only and do not necessarily reflect those of the European Union or the European Research Council Executive Agency. Neither the European Union nor the granting authority can be held responsible for them.}} P.-E. Jabin was partially supported by NSF DMS Grants 2205694 and 2219297.
The authors thank the  referees for their careful comments and suggestions, which helped improve and correct the manuscript.

\appendix
\section{Weak duality solutions}\label{sec:dual-sol}

In this appendix, we introduce a notion of ``weak duality solution'' for the Liouville equation, which is the natural notion of solution for our duality approach to mean field. In a nutshell, weak duality solutions are required to be in duality with \emph{some}, but possibly not all, bounded weak solutions of the backward equation with prescribed final condition. For the second-order dynamics~\eqref{eq:Liouville} considered in the main part of the paper, we shall prove global existence whenever $K\in L^1_\loc(\Omega;\R^d)$ and $F_N^\circ\in L^1(\D^N)\cap L^p(\D^N)$ for some $p>1$; see Proposition~\ref{prop:weak-dual-sol}. We emphasize that weak duality solutions need not be distributional solutions of the Liouville equation, since the product $K(x_i-x_j)F_N$ may fail to be well-defined. Likewise, bounded renormalized solutions of the backward equation are not required to exist, and no uniqueness is asserted in general.

Of course, weak duality solutions coincide with renormalized solutions whenever the latter exist. For $\alpha=0$, this is known for instance if $K\in BV_\loc(\Omega;\R^d)$, cf.~\cite{Bouchut}, or if $K\in L^1_\loc(\Omega;\R^d)\cap BV_\loc(\Omega\setminus\{0\};\R^d)$ takes the form $K=-\nabla V$ with $V(x)\ge-C(1+|x|^2)$, cf.~\cite{Hauray04}; we refer to~\cite{LeBris-Lions} for the case $\alpha>0$.

\begin{definition}\label{WeakForm}
Let $0\le\alpha<\infty$ and $K\in L^1_\loc(\Omega;\R^d)$.
\begin{enumerate}[{\rm(i)}]
\item Given $T>0$ and $\Phi_N^T\in L^\infty(\D^N)$, we consider the backward Liouville equation on $(-\infty,T]$,
\begin{equation}\label{eq:back-Liouv-def}
\partial_t\Phi_N+\sum_{i=1}^N\Big(v_i\cdot\nabla_{x_i}\Phi_N+\frac1{N-1}\sum_{j:j\ne i}^NK(x_i-x_j)\cdot\nabla_{v_i}\Phi_N\Big)=-\alpha\sum_{i=1}^N\Delta_{v_i}\Phi_N,
\end{equation}
with final data $\Phi_N|_{t=T}=\Phi_N^T$. We say that $\Phi_N\in L^\infty((-\infty,T]\times\D^N)$ is a \emph{\bf global bounded weak solution} if
\[\Phi_N\in C_{w^*}(-\infty,T;L^\infty(\D^N)),\qquad
\|\Phi_N\|_{L^\infty((-\infty,T]\times\D^N)}\le\|\Phi_N^T\|_{L^\infty(\D^N)},\]
and if it satisfies~\eqref{eq:back-Liouv-def} in the distributional sense: for all $G\in C^\infty_c((-\infty,T]\times\D^N)$,
\begin{multline*}
\qquad\int_{-\infty}^T\int_{\D^N}(\partial_tG)\Phi_N-\int_{\D^N}G(T)\Phi_N^T\\
+\sum_{i=1}^N\int_{-\infty}^T\int_{\D^N}\Big(v_i\cdot\nabla_{x_i}G+\frac1{N-1}\sum_{j:j\ne i}^NK(x_i-x_j)\cdot\nabla_{v_i}G\Big)\Phi_N
=\alpha\sum_{i=1}^N\int_{-\infty}^T\int_{\D^N}(\Delta_{v_i}G)\Phi_N.
\end{multline*}

\item
Given $F_N^\circ\in L^1(\D^N)$, we say that $F_N\in L^\infty_\loc(\R^+;L^1(\D^N))$ is a \emph{\bf global weak duality solution} of the Liouville equation~\eqref{eq:Liouville} with initial data $F_N^\circ$ if
\[F_N\in C(\R^+;wL^1_\loc(\D^N))\]
and if, for every $T>0$ and every $\Phi_N^T\in L^\infty(\D^N)$, there exists a global bounded weak solution $\Phi_N$ of~\eqref{eq:back-Liouv-def} on $(-\infty,T]$, with final data $\Phi_N^T$, such that
\begin{equation}\label{eq:weak-duality-def}
\qquad\int_{\D^N}\Phi_N^T F_N(T)=\int_{\D^N}\Phi_N(0)F_N^\circ.
\end{equation}
Both sides are pointwise well-defined thanks to the time continuity of $F_N$ and $\Phi_N$.
\end{enumerate}
Moreover, when the initial data $F_N^\circ$ is symmetric under permutations of the particle labels, we call $F_N$ a \emph{\bf symmetric} weak duality solution if $F_N(t)$ is symmetric for all~$t$ and if, for every $T$ and symmetric final data $\Phi_N^T$, the dual solution in~{\rm(ii)} can be chosen symmetric.
\end{definition}

As the following result shows, global weak duality solutions always exist whenever $K\in L^1_\loc(\Omega;\R^d)$ and $F_N^\circ\in L^1(\D^N)\cap L^p(\D^N)$ for some $p>1$.

\begin{proposition}\label{prop:weak-dual-sol}
Let
\[0\le\alpha<\infty,\qquad K\in L^1_\loc(\Omega;\R^d),\qquad F_N^\circ\in L^1(\D^N)\cap L^p(\D^N),\]
for some $1<p\le\infty$. Then there is a global weak duality solution
\[F_N\in L^\infty(\R^+;L^1(\D^N)\cap L^p(\D^N))\]
of the Liouville equation~\eqref{eq:Liouville} with initial data~$F_N^\circ$. If $F_N^\circ\ge0$, the solution can be chosen nonnegative. If $F_N^\circ$ is symmetric, it can be chosen symmetric. If in addition $K\in L^{p'}_\loc(\Omega;\R^d)$ with $1/p+1/p'=1$, then there exists such a global weak duality solution that further satisfies~\eqref{eq:Liouville} in the distributional sense.
\end{proposition}

\begin{proof}
We proceed by approximation. Let $K_\eps\in C^\infty_b(\Omega;\R^d)$ and $F_{N,\eps}^\circ\in L^1(\D^N)\cap L^\infty(\D^N)$ be such that $K_\eps\to K$ in $L^1_\loc(\Omega;\R^d)$ and $F_{N,\eps}^\circ\to F_N^\circ$ in $L^1(\D^N)\cap L^p(\D^N)$ as $\e\downarrow0$. If $F_N^\circ\ge0$, the approximation~$F_{N,\eps}^\circ$ can be chosen nonnegative; if~$F_N^\circ$ is symmetric, it can be chosen symmetric. For fixed $\e$, by regularity of $K_\e$ and boundedness of $F_{N,\e}^\circ$, there exists a unique global weak solution
\[F_{N,\eps}\in L^\infty(\R^+;L^1(\D^N)\cap L^\infty(\D^N))\]
of the regularized Liouville equation, cf.~e.g.~\cite{DiLi}. Since the phase-space drift is divergence-free, we have for all $1\le q\le\infty$ and $t\ge0$,
\begin{equation}\label{eq:bnd-FNeps}
\|F_{N,\eps}(t)\|_{L^q(\D^N)}\le\|F_{N,\eps}^\circ\|_{L^q(\D^N)}.
\end{equation}
Positivity and symmetry are preserved as well.
By weak compactness, up to a subsequence as $\eps\downarrow0$, we deduce that $F_{N,\eps}$ converges weakly-* locally in time and weakly locally in space to some
\[F_N\in L^\infty(\R^+;L^1(\D^N)\cap L^p(\D^N)),\]
with the bounds inherited from~\eqref{eq:bnd-FNeps}. In the particular case when $K\in L^{p'}_\loc$ with $1/p+1/p'=1$, we can pass to the limit in the weak formulation of the regularized equation and conclude that $F_N$ satisfies~\eqref{eq:Liouville} in the distributional sense.

We turn to the duality property. Let $T>0$ be fixed, as well as some $\Phi_N^T\in L^\infty(\D^N)$ with compact support. For fixed~$\eps$, there exists a unique global bounded weak solution $\Phi_{N,\eps}$ of the regularized backward equation on $(-\infty,T]$ with final data $\Phi_N^T$, and it satisfies for all $t\le T$,
\begin{equation}\label{eq:bnd-PhiNeps}
\|\Phi_{N,\eps}(t)\|_{L^\infty(\D^N)}\le\|\Phi_N^T\|_{L^\infty(\D^N)}.
\end{equation}
If $\Phi_N^T$ is symmetric, uniqueness of the regularized problem implies that $\Phi_{N,\eps}$ is symmetric. By weak-* compactness, up to a subsequence, we deduce that $\Phi_{N,\eps}$ converges weakly-* to some~$\Phi_N$ in~$L^\infty((-\infty,T]\times\D^N)$, and the latter satisfies the weak formulation of~\eqref{eq:back-Liouv-def} with final data $\Phi_N^T$. Moreover, from~\eqref{eq:bnd-PhiNeps} and the equation, $(\partial_t\Phi_{N,\eps})_\eps$ is bounded in $L^\infty_\loc(-\infty,T;W^{-2,1}_\loc(\D^N))$. By the Aubin-Lions-Simon lemma, this entails that $(\Phi_{N,\e})_\e$ is precompact, for instance, in $C_\loc(-\infty,T;W^{-1,\infty}_\loc(\D^N))$. As $W^{-1,\infty}=(W^{1,1})^*$ and as $W^{1,1}$ is dense in $L^1$, we conclude that $(\Phi_{N,\e})_\e$ is actually precompact in $C_\loc(-\infty,T;w^*L^\infty_\loc(\D^N))$. This proves the desired time continuity of the limit $\Phi_N$, which is thus a global bounded weak solution in the sense of Definition~\ref{WeakForm}.

For fixed $\e$, the forward and backward regularized solutions satisfy the following duality formula, cf.~e.g.~\cite[Section~II.4]{DiLi}, for all $t\in[0,T]$,
\begin{equation}\label{eq:duality-epsilon}
\int_{\D^N}\Phi_N^TF_{N,\e}(t)=\int_{\D^N}\Phi_{N,\e}(T-t)F_{N,\e}^\circ.
\end{equation}
Recalling that $(\Phi_{N,\e})_\e$ is precompact in $C_{\loc}(-\infty,T;w_*L^\infty(\D^N))$, that $(F_{N,\e}^\circ)_\e$ converges strongly in $L^1(\D^N)$, that $(F_{N,\e})_\e$ converges weakly-* in~$L^\infty(\R^+;L^1_\loc(\D^N)\cap L^p(\D^N))$, and that the choice of the compactly supported final data~$\Phi_N^T\in L^\infty(\D^N)$ is arbitrary, this identity~\eqref{eq:duality-epsilon} entails that $(F_{N,\e})_\e$ is precompact in $C_w(\R^+;L^1_\loc(\D^N)\cap L^p(\D^N))$, thus proving the desired time continuity of the limit $F_N$. Moreover, passing to the limit in~\eqref{eq:duality-epsilon}, at~$t=T$, we deduce for any compactly supported $\Phi_N^T\in L^\infty(\D^N)$,
\begin{equation*}
\int_{\D^N}\Phi_N^TF_{N}(T)=\int_{\D^N}\Phi_{N}(0)F_{N}^\circ.
\end{equation*}

It remains to extend this identity to every $\Phi_N^T\in L^\infty(\D^N)$ without compact support restriction. To this aim, given~$\Phi_N^T\in L^\infty(\D^N)$, choose a sequence of cutoffs $\chi_R\uparrow1$ and, for every $R$, consider a bounded weak solution $\Phi_{N,R}\in L^\infty((-\infty,T]\times\D^N)$ of the backward equation on $(-\infty,T]$ with final data $\chi_R\Phi_N^T$, as constructed above, such that the duality relation holds,
\begin{equation*}
\int_{\D^N}\chi_R\Phi_N^TF_{N}(T)=\int_{\D^N}\Phi_{N,R}(0)F_{N}^\circ.
\end{equation*}
Since $(\chi_R\Phi_{N}^T)_R$ is bounded in $L^\infty(\D^N)$, the same argument as above ensures that $(\Phi_{N,R})_R$ is precompact in $C_{\loc}(-\infty,T;w^*L^\infty_\loc(\D^N))$. This allows to pass to the limit in this identity and the conclusion follows.
\end{proof}

\begin{remark}[Possible loss of mass]\label{rem:mass-loss}
A weak duality solution in the above sense need not remain a probability density, even if $F_N^\circ$ is one initially. Indeed, choosing $\Phi_N^T\equiv1$, the definition yields
\[\int_{\D^N}F_N(T)=\int_{\D^N}\Phi_N(0)F_N^\circ,\]
for some bounded weak solution $\Phi_N$ of the backward equation with final data $\Phi_N^T\equiv1$. Using that $\|\Phi_N(0)\|_{L^\infty(\D^N)}\le1$, we deduce
\[\int_{\D^N}F_N(T)\le\int_{\D^N}F_N^\circ.\]
Equality is not guaranteed since, in the absence of uniqueness for the backward equation, the solution~$\Phi_N$ selected by the duality relation need not coincide with the constant solution $\Phi_N\equiv1$.

This possible loss of mass is also apparent in the approximation procedure used in the proof of Proposition~\ref{prop:weak-dual-sol} above: the compactness available for $F_{N,\e}$ is only local in phase space, so that mass may escape to infinity in the regularization limit $\e\downarrow0$.
This is consistent with possible blow-up of the underlying particle dynamics: if a positive fraction of regularized trajectories escapes to infinity in phase space in finite time, the corresponding mass disappears in the local limit.
\end{remark}

\subsection{First-order dynamics}\label{app:1storder-dualsol}
We can similarly introduce weak duality solutions for the Liouville equation~\eqref{eq:Liouville-1st} associated with first-order systems as considered in Theorem~\ref{th:main-1st}. There is, however, an important difference with the second-order setting above. Indeed, if $K$ is compressible, that is, $\Div(K)\ne0$, then the a priori $L^p$ bounds on~$F_N$ used in Proposition~\ref{prop:weak-dual-sol} are no longer available in general. In fact, even starting from an absolutely continuous distribution $F_N^\circ$, concentration on lower-dimensional sets may occur in finite time, in link with sticky collisions in the particle dynamics; see e.g.~\cite{CaDiFiLaSl11,CaJaLaVa16,JaVa16}. In this setting, we thus have to allow weak duality solutions to take values in the space $\Mcal(\Omega^N)$ of finite measures on $\Omega^N$. It turns out that the definition of duality solution is naturally adapted to this setting and that global existence can still be ensured, provided that initial data $F_N^\circ$ are in $L^1(\Omega^N)$. On top of $K\in L^1_\loc(\Omega;\R^d)$, note that the assumption $\Div(K)\in L^1_\loc(\Omega)$ is now further needed to make sense of the weak formulation of the backward equation.

Again, weak duality solutions coincide with renormalized solutions whenever the latter exist: for $\alpha=0$, this is known for instance if $K\in BV_\loc(\Omega;\R^d)$ and $\Div(K)\in L^\infty_\loc(\Omega)$, cf.~\cite{Ambrosio-04}, and we refer to~\cite{LeBris-Lions} for the case $\alpha>0$.

\begin{definition}\label{WeakForm-1st}
Let $0\le \alpha<\infty$ and $K\in L^1_\loc(\Omega;\R^d)$ with $\Div(K)\in L^1_\loc(\Omega)$.
\begin{enumerate}[{\rm(i)}]
\item Given $T>0$ and $\Phi_N^T\in C_b(\Omega^N)$, we consider the backward Liouville equation on $(-\infty,T]$,
\begin{equation}\label{eq:back-Liouv-def-1st}
\partial_t\Phi_N
+\frac1{N-1}\sum_{i\ne j}^N
K(x_i-x_j)\cdot\nabla_{x_i}\Phi_N
=
-\alpha\sum_{i=1}^N\Delta_{x_i}\Phi_N,
\end{equation}
with final data $\Phi_N|_{t=T}=\Phi_N^T$.
We say that $\Phi_N$ is a \emph{\bf global bounded weak solution} if
\[\Phi_N\in C_{w^*}(-\infty,T;L^\infty(\Omega^N)),\qquad
\sup_{t\le T}\|\Phi_N(t)\|_{L^\infty(\D^N)}\le\|\Phi_N^T\|_{L^\infty(\D^N)},\]
and if it satisfies~\eqref{eq:back-Liouv-def-1st} in the distributional sense.
\smallskip\item Given $F_N^\circ\in L^1(\Omega^N)$, we say that $F_N$ is a \emph{\bf global measure-valued weak duality solution} of~\eqref{eq:Liouville-1st} if
\[F_N\in C_{w^*}(\R^+;\Mcal(\Omega^N)),\qquad
\sup_{t\ge0}\|F_N(t)\|_{\Mcal(\Omega^N)}\le \|F_N^\circ\|_{L^1(\Omega^N)},\]
and if, for every $T>0$ and every $\Phi_N^T\in C_b(\Omega^N)$, there exists a global bounded weak solution $\Phi_N$ of~\eqref{eq:back-Liouv-def-1st} on $(-\infty,T]$, with final data $\Phi_N^T$, such that
\[\int_{\Omega^N}\Phi_N^T dF_N(T)=\int_{\Omega^N}\Phi_N(0)F_N^\circ.\]
\end{enumerate}
For symmetric data $F_N^\circ$, we call $F_N$ a \emph{\bf symmetric} weak duality solution if $F_N(t)$ is symmetric for all~$t$ and if, for every $T$ and symmetric final data $\Phi_N^T$, the dual solution in~(ii) can be chosen symmetric.
\end{definition}

We show existence of such global measure-valued weak duality solution whenever $F_N^\circ\in L^1(\Omega^N)$. We further show that no concentration can occur when $\Div(K)$ is bounded from below, and that no loss of mass at infinity can occur when $K$ is odd.

\begin{proposition}\label{prop:weak-dual-sol-1st}
Let
\[0\le\alpha<\infty,\qquad K\in L^1_\loc(\Omega;\R^d),\qquad\Div(K)\in L^1_\loc(\Omega),\qquad F_N^\circ\in L^1(\Omega^N).\]
Then there is a global measure-valued weak duality solution $F_N$ of~\eqref{eq:Liouville-1st} with initial data~$F_N^\circ$. If~$F_N^\circ\ge0$, the solution can be chosen nonnegative. If~$F_N^\circ$ is symmetric, it can be chosen symmetric. Moreover, finer properties can be ensured under additional structural assumptions on $K$:
\begin{enumerate}[(i)]
\item \emph{No concentration.} If $\Div(K)\ge-C$ and $F_N^\circ\in L^p(\Omega^N)$ for some $1<p\le\infty$, then the weak duality solution can be chosen in $L^\infty_\loc(\R^+;L^p(\Omega^N))$ such that for all $t\ge0$,
\[\|F_N(t)\|_{L^p(\Omega^N)}\le e^{CN(1-1/p)t}\|F_N^\circ\|_{L^p(\Omega^N)}.\]
If in addition $K\in L^{p'}_\loc(\Omega)$ with $1/p+1/p'=1$, then this solution can further be chosen to satisfy~\eqref{eq:Liouville-1st} in the distributional sense.
\smallskip\item \emph{No loss of mass at infinity.} For $\Omega=\R^d$, if $K$ is odd and satisfies $\sup_{|x|\ge1}|K(x)|<\infty$, then the weak duality solution can be chosen mass-preserving: $\int_{\Omega^N}F_N(t)=\int_{\Omega^N}F_N^\circ$ for all $t\ge0$. In particular, if $F_N^\circ$ is a probability measure, the solution can be chosen to remain a probability measure for all times. For $\Omega=\T^d$, this holds without any requirement on $K$.
\end{enumerate}
\end{proposition}

\begin{proof}
We start with the main part of the statement and briefly indicate the differences with the proof of Proposition~\ref{prop:weak-dual-sol}. Let $K_\e$ be a smooth bounded regularization of $K$ with $K_\e\to K$ and $\Div(K_\e)\to\Div(K)$ in $L^1_\loc(\Omega)$, and let $F_{N,\e}^\circ$ be a smooth approximation of $F_N^\circ$. For each $\e>0$, the regularized forward equation has a unique global solution $F_{N,\e}$ with
\[\int_{\Omega^N}|F_{N,\e}(t)|=\int_{\Omega^N}|F_{N,\e}^\circ|.\]
Due to possible compressibility, no uniform $L^p$ estimate with $p>1$ is available in general, but this gives at least uniform boundedness of $(F_{N,\e})_\e$ in $L^\infty(\R^+;\Mcal(\Omega^N))$, hence vague compactness in that space.
On the other hand, on the dual side, the regularized backward solutions satisfy for all $t\le T$,
\[\|\Phi_{N,\e}(t)\|_{L^\infty(\Omega^N)}\le\|\Phi_N^T\|_{L^\infty(\Omega^N)}.\]
Using the local $L^1$ bounds on $K_\e$ and $\Div(K_\e)$, together with this uniform $L^\infty$ bound on $(\Phi_{N,\e})_\e$, the equation implies that $(\partial_t\Phi_{N,\e})_\e$ is bounded in $L^\infty_\loc(-\infty,T;W^{-2,1}_\loc(\Omega^N))$. The same Aubin-Lions-Simon argument as in Proposition~\ref{prop:weak-dual-sol} then shows that $(\Phi_{N,\e})_\e$ is precompact in $C_{\loc}(-\infty,T;w^*L^\infty_\loc(\Omega^N))$.
Up to a subsequence, we may then pass to the limit in the regularized duality identity
\[\int_{\Omega^N}\Phi_N^T F_{N,\e}(T)=\int_{\Omega^N}\Phi_{N,\e}(0)F_{N,\e}^\circ,\]
which then yields the claimed measure-valued weak duality solution.

We turn to the proof of~(i). If $\Div(K)\ge-C$, the regularizations can be chosen so that $\Div(K_\e)\ge-C$. The following a priori estimate then holds for the regularized solutions, for all $1\le p<\infty$,
\[\frac{d}{dt}\|F_{N,\e}\|_{L^p(\Omega^N)}^p\le CN(p-1)\|F_{N,\e}\|_{L^p(\Omega^N)}^p,\]
hence
\[\|F_{N,\e}(t)\|_{L^p(\Omega^N)}\le e^{CN(1-1/p)t}\|F_{N,\e}^\circ\|_{L^p(\Omega^N)},\]
and these estimates pass to the limit. If in addition $K\in L^{p'}_\loc(\Omega;\R^d)$, we can pass to the limit in the forward equation and the limit $F_N$ is then a distributional solution.

It remains to prove~(ii). For $\Omega=\T^d$, the result holds directly by compactness, so we only need to consider the case $\Omega=\R^d$.
By the additional assumptions on $K$, we can choose the regularizations $K_\e$ to be odd and such that $|K_\e(x)|\le C$ for $|x|\ge1$. For $\e>0$, the regularized forward solution can be represented as $F_{N,\e}(t)=\E_B[X_\e(t)_*F_{N,\e}^\circ]$, where $X_\e(t,X)=(X_{i,\e}(t,X))_i$ is the stochastic flow on~$(\R^d)^N$ associated with
\begin{equation}\label{eq:SDE-first-regularized}
dX_{i,\e}=\frac1{N-1}\sum_{j:j\ne i}^NK_\e(X_{i,\e}-X_{j,\e})dt+\sqrt{2\alpha}dB_i,\qquad 1\le i\le N.
\end{equation}
We claim that for every $T<\infty$ the latter satisfies
\begin{equation}\label{eq:cluster-tightness}
\max_{1\le i\le N}\sup_{0\le t\le T}|X_{i,\e}(t,X)|
\le C_N\Big(\max_i|X_i|+\sqrt{\alpha}M_T+CT+1\Big),\quad M_T:=\max_i\sup_{0\le t\le T}|B_i(t)|,
\end{equation}
for some $C_N<\infty$ independent of $\e$. This is a consequence of the oddness of the interaction: it follows for instance by a standard clustering argument as in~\cite[Section~2]{FoMa07}.
As a consequence, we can bound for all $t\in[0,T]$ and $L\ge1$, setting $L':=C_N(2L+CT+1)$,
\[\Big|\int_{\max_i|X_i|>L'}F_{N,\e}(t,X)dX\Big|\le \int_{\max_i|X_i|\ge L} |F_{N,\e}^\circ(X)|dX+\Pm_B[\sqrt\alpha M_T>L]\int_{(\R^d)^N}|F_{N,\e}^\circ|.\]
Since $F_{N,\e}^\circ\to F_N^\circ$ in $L^1((\R^d)^N)$ and since $M_T<\infty$ a.s., we deduce
\[\limsup_{\e>0}\Big|\int_{\max_i|X_i|>L}F_{N,\e}(t,X)dX\Big|\to0\quad\text{as $L\uparrow\infty$.}\]
Combined with the vague convergence of $F_{N,\e}(t)$, this implies $\int_{(\R^d)^N} F_{N,\e}(t)\to\int_{(\R^d)^N} F_{N}(t)$. As the regularized solution preserves the mass, this concludes the proof.
\end{proof}

\subsection{More general phase-space interactions}\label{sec:dual-sol-general}
The same construction extends to genuinely phase-space dependent interactions as considered in Section~\ref{sec:gen}. Consider the Liouville equation~\eqref{eq:Liouville-general-K} for some general pairwise interaction $K:\D^2\to\R^d$. As for the first-order dynamics, sticky collisions are possible when $\Div_v(K)(z,z')\not\equiv0$, hence we have to allow weak duality solutions to take values in general in $\Mcal(\Dd^N)$. We omit the details for shortness.

\begin{definition}
Let $0\le\alpha<\infty$ and $K\in L^1_\loc(\D^2;\R^d)$ with $\Div_v(K)\in L^1_\loc(\D^2)$.
\begin{enumerate}[{\rm(i)}]
\item Given $T>0$ and $\Phi_N^T\in C_b(\D^N)$, we consider the backward Liouville equation on $(-\infty,T]$,
\begin{equation}\label{eq:back-Liouv-def-gen}
\partial_t\Phi_N+\sum_{i=1}^Nv_i\cdot\nabla_{x_i}\Phi_N+\frac1{N-1}\sum_{i\ne j}^NK(z_i,z_j)\cdot\nabla_{v_i}\Phi_N=-\alpha\sum_{i=1}^N\Delta_{v_i}\Phi_N,
\end{equation}
with final data $\Phi_N|_{t=T}=\Phi_N^T$.
We say that $\Phi_N$ is a \emph{\bf global bounded weak solution} if
\[\Phi_N\in C_{w^*}(-\infty,T;L^\infty(\D^N)),\qquad
\sup_{t\le T}\|\Phi_N(t)\|_{L^\infty(\D^N)}\le\|\Phi_N^T\|_{L^\infty(\D^N)},\]
and if it satisfies~\eqref{eq:back-Liouv-def-gen} in the distributional sense.
\smallskip\item Given $F_N^\circ\in L^1(\D^N)$, we say that $F_N$ is a \emph{\bf global measure-valued weak duality solution} of~\eqref{eq:Liouville-general-K} if
\[F_N\in C_{w^*}(\R^+;\Mcal(\D^N)),\qquad
\sup_{t\ge0}\|F_N(t)\|_{\Mcal(\D^N)}\le \|F_N^\circ\|_{L^1(\D^N)},\]
and if, for every $T>0$ and every $\Phi_N^T\in C_b(\D^N)$, there exists a global bounded weak solution $\Phi_N$ of~\eqref{eq:back-Liouv-def-1st} on $(-\infty,T]$, with final data $\Phi_N^T$, such that
\[\int_{\D^N}\Phi_N^T dF_N(T)=\int_{\D^N}\Phi_N(0)F_N^\circ.\]
\end{enumerate}
For symmetric data $F_N^\circ$, we call $F_N$ a \emph{\bf symmetric} weak duality solution if $F_N(t)$ is symmetric for all~$t$ and if, for every $T$ and symmetric final data $\Phi_N^T$, the dual solution in~(ii) can be chosen symmetric.
\end{definition}

\begin{proposition}\label{prop:weak-dual-sol-general-K}
Let
\[0\le\alpha<\infty,\qquad K\in L^1_\loc(\Dd^2;\R^d),\qquad\Div_v(K)\in L^1_\loc(\D^2),\qquad F_N^\circ\in L^1(\D^N).\]
Then there is a global measure-valued weak duality solution $F_N$ of~\eqref{eq:Liouville-general-K} with initial data~$F_N^\circ$. If $F_N^\circ\ge0$, the solution can be chosen nonnegative. If $F_N^\circ$ is symmetric, it can be chosen symmetric.
Moreover, if $\Div_v(K)\ge-C$ and if $F_N^\circ\in L^p(\D^N)$ for some $1<p\le\infty$, then the solution can be chosen in~$L^\infty_\loc(\R^+;L^p(\D^N))$.
\end{proposition}

\section{Explicit dual hierarchy}\label{sec:app}
For completeness, we record here the full hierarchy satisfied by dual correlations $\{C_{N,n}\}_{n}$, beyond the projected formulation used in Lemma~\ref{propCn}. Equivalently, we identify explicitly the remainder $R_{N,n}$ appearing in~\eqref{Ck}. This more precise formulation is not needed for the proof of the mean-field limit in the present $L^2$ setting, but it is useful for quantitative error estimates, cf.~\cite{KhJa}, and for finer applications of the duality method beyond the square-integrable regime~\cite{DuJa2}.

\begingroup\allowdisplaybreaks
\begin{lemma}\label{lem:BBGKY}
Under the assumptions of Lemma~\ref{propCn}, the remainder $R_{N,n}$ appearing in~\eqref{Ck} is given, for all $0\le n\le N$, by
\begin{eqnarray}
R_{N,n}
&=&\frac1{N-1}\sum_{i\ne j}^n(K\ast f)(x_i)\cdot\nabla_{v_i}C_{N,n-1}(z_{[n]\setminus\{j\}})
\nonumber\\
&-&\frac1{N-1}\sum_{i\ne j}^n\int_\D K_f(z_*,z_j)C_{N,n-1}(z_{[n]\setminus\{i,j\}},z_*)f(z_*)\,dz_*
\nonumber\\
&+&\frac1{N-1}\sum_{i\ne j}^n\int_\D K(x_i-x_*)\cdot\nabla_{v_i}C_{N,n}(z_{[n]\setminus\{j\}},z_*)f(z_*)\,dz_*
\nonumber\\
&-&\frac1{N-1}\sum_{i\ne j}^n\int_\D K_f(z_*,z_j)C_{N,n}(z_{[n]\setminus\{i\}},z_*)f(z_*)\,dz_*
\nonumber\\
&+&\frac1{N-1}\sum_{i\ne j}^n\int_{\D^2}K_f(z_*,z_*')C_{N,n}(z_{[n]\setminus\{i,j\}},z_*,z_*')f(z_*)f(z_*')\,dz_*dz_*'
\nonumber\\
&-&2\frac{N-n}{N-1}\sum_{i=1}^n\int_{\D^2}K_f(z_*,z_*')C_{N,n+1}(z_{[n]\setminus\{i\}},z_*,z_*')f(z_*)f(z_*')\,dz_*dz_*'.\label{eq:explicit-RNn}
\end{eqnarray}
\end{lemma}

\begin{proof}
We start from the equations~\eqref{eq:BBGKY-MNn0} for marginals. In order to derive corresponding equations for correlations, we appeal to the inversion formula~\eqref{eq:def-CNn-M}; this leads us to
\begin{equation*}
\partial_tC_{N,n}+\sum_{i=1}^nv_i\cdot\nabla_{x_i}C_{N,n}+\alpha\sum_{i=1}^n\Delta_{v_i}C_{N,n}
\,=\,T_1+T_2+T_3+T_4,
\end{equation*}
where we have set
\begin{eqnarray*}
T_i\,:=\,\sum_{k=0}^n(-1)^{n-k}\sum_{\sigma\in P_k^n}T_{i,k}'(z_\sigma),
\end{eqnarray*}
in terms of
\begin{eqnarray*}
T_{1,k}'(z_\sigma)&:=&-\frac1{N-1}\sum_{i,j\in \sigma\atop i\ne j}K(x_i-x_j)\cdot\nabla_{v_i}M_{N,k}(z_\sigma),\\
T_{2,k}'(z_\sigma)&:=&\frac{N-k}{N-1}\sum_{j\in\sigma}\int_{\D}K_f(z_*,z_j)M_{N,k+1}(z_{\sigma},z_*)\,f(z_*)\,dz_*,\\
T_{3,k}'(z_\sigma)&:=&-\frac{N-k}{N-1}\sum_{i\in\sigma}\int_{\D}K(x_i-x_*)\cdot\nabla_{v_i}M_{N,k+1}(z_\sigma,z_*)\,f(z_*)\,dz_*,\\
T_{4,k}'(z_\sigma)&:=&\frac{(N-k)(N-k-1)}{N-1}\int_{\D^2}K_f(z_*,z_*')\,M_{N,k+2}(z_\sigma,z_*,z_*')\,f(z_*)f(z_*')\,dz_*dz_*'.
\end{eqnarray*}
Now inserting the cluster expansion~\eqref{eq:clusterexpPhin} to express each term in terms of correlations, we are left with a finite inclusion-exclusion computation.
To this aim, we shall repeatedly use the elementary finite-difference identity
\begin{equation}\label{eq:diff-id}
\sum_{k=l+r}^n(-1)^{n-k}\binom{n-l-r}{k-l-r}q(k)=\Delta^{n-l-r}q(l+r),
\end{equation}
where $\Delta q(k):=q(k+1)-q(k)$. Since the coefficients appearing in the above expressions are polynomials in $k$ of degree at most two, all finite differences of order larger than two vanish; this immediately explains why only neighboring correlation levels can occur in the resulting hierarchy.

We examine the different terms $\{T_i\}_{1\le i\le 4}$ separately.
First, for $T_1$, reorganizing the sums, we find
\begin{eqnarray*}
T_1&=&-\frac1{N-1}\sum_{k=0}^n(-1)^{n-k}\sum_{\sigma\in P_k^n}\sum_{i,j\in \sigma\atop i\ne j}K(x_i-x_j)\cdot\nabla_{v_i}M_{N,k}(z_\sigma)\\
&=&-\frac1{N-1}\sum_{k=0}^n(-1)^{n-k}\sum_{\sigma\in P_k^n}\sum_{i,j\in \sigma\atop i\ne j}K(x_i-x_j)\cdot\nabla_{v_i}\sum_{l=0}^k\sum_{\tau\in P^\sigma_l}C_{N,l}(z_\tau)\\
&=&-\frac1{N-1}\sum_{i\ne j}^nK(x_i-x_j)\cdot\nabla_{v_i}\sum_{l=0}^n\sum_{\tau\in P^n_l}C_{N,l}(z_\tau)
\sum_{k=l}^n(-1)^{n-k}\,\sharp\big\{\sigma\in P_k^n:\sigma\supset\tau\cup\{i,j\}\big\},
\end{eqnarray*}
and thus, distinguishing the cases whether $j\in\tau$ or $j\notin\tau$, noting that the contribution vanishes if $i\notin\tau$, and computing the cardinalities,
\begin{multline*}
T_1\,=\,-\frac1{N-1}\sum_{i\ne j}^nK(x_i-x_j)\cdot\nabla_{v_i}
\bigg(\sum_{l=0}^n\sum_{\tau\in P^n_l\atop\tau\ni j}C_{N,l}(z_\tau)
\sum_{k=l}^n(-1)^{n-k}\binom{n-l}{k-l}\\
+\sum_{l=0}^n\sum_{\tau\in P^n_l\atop\tau\not\ni j}C_{N,l}(z_\tau)
\sum_{k=l+1}^n(-1)^{n-k}\binom{n-l-1}{k-l-1}\bigg).
\end{multline*}
Now using the binomial identity $\sum_{k=0}^p(-1)^{p-k}\binom pk=1_{p=0}$, which is a particular case of~\eqref{eq:diff-id}, we deduce
\begin{equation*}
T_1\,=\,-\frac1{N-1}\sum_{i\ne j}^nK(x_i-x_j)\cdot\nabla_{v_i}
\Big(C_{N,n}(z_{[n]})
+C_{N,n-1}(z_{[n]\setminus\{j\}})\Big).
\end{equation*}
Next, for the second term $T_2$, using~\eqref{eq:cancel-Vf}, we can similarly write
\begin{eqnarray*}
T_2&=&\sum_{k=0}^{n}(-1)^{n-k}\sum_{\sigma\in P_k^n}\frac{N-k}{N-1}\sum_{j\in\sigma}\int_\D K_f(z_*,z_j)M_{N,k+1}(z_\sigma,z_*)\,f(z_*)\,dz_*\\
&=&\sum_{k=0}^{n}(-1)^{n-k}\sum_{\sigma\in P_k^n}\frac{N-k}{N-1}\sum_{j\in\sigma}\int_\D K_f(z_*,z_j)\sum_{l=0}^{k}\sum_{\tau\in P^\sigma_l}C_{N,l+1}(z_\tau,z_*)\,f(z_*)\,dz_*\\
&=&\sum_{j=1}^n\sum_{l=0}^{n}\sum_{\tau\in P^n_l}\int_\D K_f(z_*,z_j)C_{N,l+1}(z_\tau,z_*)\,f(z_*)\,dz_*\sum_{k=l}^{n}(-1)^{n-k}\frac{N-k}{N-1}\,\sharp\big\{\sigma\in P_k^n:\sigma\supset\tau\cup\{j\}\big\}.
\end{eqnarray*}
Distinguishing the cases whether $j\in\tau$ or $j\notin\tau$, computing the cardinalities, and now using~\eqref{eq:diff-id} in the form $\sum_{k=l}^{n}(-1)^{n-k}\binom{n-l}{k-l}\frac{N-k}{N-1}=1_{l=n}\frac{N-n}{N-1}-1_{l=n-1}\frac{1}{N-1}$, we obtain
\begin{eqnarray*}
T_2
&=&\frac{N-n}{N-1}\sum_{j=1}^n\int_\D K_f(z_*,z_j)C_{N,n+1}(z_{[n]},z_*)\,f(z_*)\,dz_*\\
&-&\frac{1}{N-1}\sum_{i\ne j}^n\int_\D K_f(z_*,z_j)C_{N,n}(z_{[n]\setminus\{i\}},z_*)\,f(z_*)\,dz_*\\
&+&\frac{N-n}{N-1}\sum_{j=1}^n\int_\D K_f(z_*,z_j)C_{N,n}(z_{[n]\setminus\{j\}},z_*)\,f(z_*)\,dz_*\\
&-&\frac{1}{N-1}\sum_{i\ne j}^n\int_\D K_f(z_*,z_j)C_{N,n-1}(z_{[n]\setminus\{i,j\}},z_*)\,f(z_*)\,dz_*.
\end{eqnarray*}
Similar computations for $T_3$ and $T_4$ yield
\begin{eqnarray*}
T_3&=&-\frac{N-n}{N-1}\sum_{i=1}^n\int_\D K(x_i-x_*)\cdot\nabla_{v_i}C_{N,n+1}(z_{[n]},z_*)\,f(z_*)\,dz_*\\
&+&\frac{1}{N-1}\sum_{i\ne j}^n\int_\D K(x_i-x_*)\cdot\nabla_{v_i}C_{N,n}(z_{[n]\setminus\{j\}},z_*)\,f(z_*)\,dz_*\\
&-&\frac{N-n}{N-1}\sum_{i=1}^n K\ast f(x_i)\cdot\nabla_{v_i}C_{N,n}(z_{[n]})\\
&+&\frac{1}{N-1}\sum_{i\ne j}^n K\ast f(x_i)\cdot\nabla_{v_i}C_{N,n-1}(z_{[n]\setminus\{j\}}),\\
T_4&=&\frac{1}{N-1}\sum_{i\ne j}^n\int_{\D^2} K_f(z_*,z_*')\,C_{N,n}(z_{[n]\setminus\{i,j\}},z_*,z_*')\,f(z_*)f(z_*')\,dz_*dz_*'\\
&-&2\frac{N-n}{N-1}\sum_{i=1}^n\int_{\D^2} K_f(z_*,z_*')\,C_{N,n+1}(z_{[n]\setminus\{i\}},z_*,z_*')\,f(z_*)f(z_*')\,dz_*dz_*'\\
&+&\frac{(N-n)(N-n-1)}{N-1}\int_{\D^2} K_f(z_*,z_*')\,C_{N,n+2}(z_{[n]},z_*,z_*')\,f(z_*)f(z_*')\,dz_*dz_*'.
\end{eqnarray*}
Combining these different identities and moving to the left-hand side the terms already displayed in~\eqref{Ck} yields exactly~\eqref{eq:explicit-RNn}.
\end{proof}
\endgroup


\section{Mean-field regularity}\label{sec:regularity}
In this appendix, we briefly argue that the regularity requirements on the mean-field solution~$f$ in Theorem~\ref{th:main} are satisfied at least on short times for smooth enough initial data. Although largely standard, we include details for completeness.

\begin{lemma}\label{lem:reg-mf}
Assume
\[K\in L^2_\loc(\Omega;\R^d),\qquad K\in L^\infty(|x|>1),\]
and let $f_\circ\in \Pc(\Dd)\cap L^\infty(\D)$ satisfy $\sqrt{f_\circ}\in H^1(\D)$, $\langle v\rangle^{d+1}f_\circ\in L^\infty(\D)$, and, for some $m>\frac d2$,
\[\|f_\circ\|_{X_m}:=\sum_{|\beta|\le1}\Big(\|\partial^\beta_z f_\circ\|_{L^1(\D)}+\|\langle v\rangle^m\partial^\beta_z f_\circ\|_{L^2(\D)}\Big)<\infty.\]
Then there exists~$T>0$ such that the Vlasov equation~\eqref{eq:Vlasov} admits a unique solution on $[0,T]$ satisfying
\[f\in L^\infty(0,T;\Pc(\D)\cap L^\infty(\D))\quad\text{and}\quad \sup_{[0,T]}\|f\|_{X_m}<\infty.\]
In particular, it satisfies $\rho_f\in L^\infty([0,T]\times\Omega)$, $K\ast f\in L^\infty(0,T;W^{1,\infty}(\Omega))$, and
\begin{equation}\label{eq:fisher-bnd}
\sup_{[0,T]}\int_{\D}|\nabla_z\log f|^2f<\infty.
\end{equation}
\end{lemma}

\begin{proof}
For $|\beta|\le1$, the assumptions on $K$ and the choice $m>\frac d2$ ensure
\[\|\partial^\beta_xK\ast f\|_{L^\infty(\Omega)}=\|K\ast\partial^\beta_x f\|_{L^\infty(\Omega)}\lesssim \|\partial^\beta_x f\|_{L^1(\Dd)}+\|\langle v\rangle^m\partial^\beta_x f\|_{L^2(\D)},\]
hence
\[\|K\ast f\|_{W^{1,\infty}(\Omega)}\lesssim\|f\|_{X_m}.\]
A standard Picard iteration, with uniform bounds in $X_m$, then ensures that there exist $T>0$ and a unique solution $f$ with $\sup_{[0,T]}\|f\|_{X_m}<\infty$. Mass conservation and the maximum principle further imply \mbox{$\int_\Dd f(t)=1$}, \mbox{$f(t)\ge0$}, and $\|f(t)\|_{L^\infty(\Dd)}\le\|f_\circ\|_{L^\infty(\Dd)}$. Applying the same maximum principle with the polynomial velocity weight $\langle v\rangle^{d+1}$ yields a locally uniform bound on $\langle v\rangle^{d+1}f$, and hence $\rho_f\in L^\infty([0,T]\times\Omega)$.

It remains to show~\eqref{eq:fisher-bnd}. First consider the case $\alpha=0$: equation~\eqref{eq:Vlasov} is then a transport equation for~$f$ on $\Dd$ with vector field $(v,(K\ast f)(t,x))$. Since the latter is divergence-free and Lipschitz, the associated flow is bi-Lipschitz and volume-preserving, and we have $f(t)=f_{\circ}\circ\Phi_t^{-1}$. The chain rule then ensures
\[\int_\D|\nabla_z\log f(t)|^2f(t)=4\int_\D |\nabla_z\sqrt{f(t)}|^2\lesssim\int_\D|\nabla_z\sqrt{f_\circ}|^2,\]
which is finite by assumption.
Finally, in the case $\alpha>0$, a direct calculation yields
\[\frac{d}{dt}\int_\D |\nabla_z\sqrt f|^2\lesssim(1+\|\nabla K\ast f\|_{L^\infty(\Omega)})\int_\D |\nabla_z\sqrt f|^2.\]
and the claim follows by Gr\"onwall's inequality.
\end{proof}

\end{document}